\documentclass[numbook,envcountsame]{svjour3} % onecolumn (second format)
\smartqed % flush right qed marks, e.g. at end of proof
\usepackage{amsmath}
\usepackage{amssymb}
\usepackage{graphicx} % standard LaTeX graphics tool
\usepackage{sidecap}
\usepackage[mathscr]{euscript}
\newcommand{\I}{{\mathbb I}}
\newcommand{\N}{{\mathbb N}}
\newcommand{\R}{{\mathbb R}}
\newcommand{\Z}{{\mathbb Z}}
\newcommand{\cA}{{\mathcal A}}
\newcommand{\cB}{{\mathcal B}}

\newcommand{\cU}{{\mathcal U}}

\newcommand{\eF}{{\mathscr F}}
\newcommand{\eG}{{\mathscr G}}
\newcommand{\eK}{{\mathscr K}}
\newcommand{\eL}{{\mathscr L}}
\newcommand{\eN}{{\mathscr N}}
\renewcommand{\d}{\,{\mathrm d}}
\newcommand{\eps}{\varepsilon}
\newcommand{\tm}{\times}
\newcommand{\vphi}{\varphi}
\newcommand{\ocA}{\omega_\cA^+}

\newcommand{\intoo}[1]{\bigl(#1\bigr)}							% open interval
\newcommand{\intcc}[1]{\left[#1\right]}							% closed interval
\newcommand{\set}[1]{\left\{#1\right\}}							% set
\newcommand{\abs}[1]{\left|#1\right|}							% absolute value
\newcommand{\norm}[1]{\left\|#1\right\|}						% norm
\newcommand{\dist}[1]{\operatorname{dist}\intoo{#1}}			% distance
\newcommand{\dar}[1]{\operatorname{dar}\bigl(#1\bigr)}			% Darbo constant

\newcommand{\fall}{\quad\text{for all }}

\newcommand{\eref}[1]{Exam.~\ref{#1}}
\newcommand{\lref}[1]{Lemma~\ref{#1}}
\newcommand{\pref}[1]{Prop.~\ref{#1}}
\newcommand{\tref}[1]{Thm.~\ref{#1}}
\newcommand{\cref}[1]{Cor.~\ref{#1}}
\newcommand{\fref}[1]{Fig.~\ref{#1}}
\newcommand{\rref}[1]{Rem.~\ref{#1}}
\newcommand{\sref}[1]{Sect.~\ref{#1}}
\newcommand{\Cd}{C(\Omega,\R^d)}

\allowdisplaybreaks
\begin{document}

\title{Forward and pullback dynamics of nonautonomous integrodifference equations\thanks{The work of Huy Huynh has been supported by the Austrian Science Fund (FWF) under grant number P 30874-N35.}
}
\subtitle{Basic constructions}
%
%\titlerunning{Short form of title} % if too long for running head
%
\author{Huy Huynh \and Peter E.\ Kloeden \and Christian P\"otzsche}

\authorrunning{H.\ Huynh, P.E.\ Kloeden, C.\ P\"otzsche} % if too long for running head

\institute{Huy Huynh and Christian P\"otzsche\at
			Institut f\"ur Mathematik, Universit\"at Klagenfurt, 9020 Klagenfurt, Austria\\
			\email{\{pham.huynh,christian.poetzsche\}@aau.at}
			\and
			Peter E.\ Kloeden\at
			Mathematisches Institut, Universit\"at T\"ubingen, 72076 T\"ubingen, Germany\\
			\email{kloeden@na.uni-tuebingen.de}}

\dedication{Dedicated to the memory of Russell Johnson}

\date{Received: \today / Accepted: date}
\maketitle
\begin{abstract}
	In theoretical ecology, models describing the spatial dispersal and the temporal evolution of species having non-overlapping generations are often based on integrodifference equations. For various such applications the environment has an aperiodic influence on the models leading to nonautonomous integrodifference equations. In order to capture their long-term behaviour comprehensively, both pullback and forward attractors, as well as forward limit sets are constructed for general infinite-dimensional nonautonomous dynamical systems in discrete time. While the theory of pullback attractors, but not their application to integrodifference equations, is meanwhile well-established, the present novel approach is needed in order to understand their future behaviour. 
\end{abstract}
\section{Introduction}
Integrodifference equations not only occur as temporal discretisations of integro\-differential equations or as time-$1$-maps of evolutionary differential equations, but are of interest in themselves. First and foremost, they are a popular tool in theoretical ecology to describe the dispersal of species having non-overlapping generations (see, for instance, \cite{lutscher:19} or \cite{kot:schaeffer:86,volkov:lui:07,jacobsen:jin:lewis:15}). While the theory of Urysohn or Hammerstein integral equations is now rather classical \cite{martin:76}, both numerically and analytically, our goal is here to study their iterates from a dynamical systems perspective. This means one is interested in the long term behaviour of recursions based on a fixed nonlinear integral operator. In applications, the iterates for instance represent the spatial distribution of interacting species over a habitat. One of the central questions in this context is the existence and structure of an attractor. These invariant and compact sets attract bounded subsets of an ambient state space $X$ and fully capture the asymptotics of an autonomous dynamical system \cite{hale:88,sell:you:02}. The dynamics inside the attractor can be very complicated and even chaotic \cite{day:junge:mischaikow:04}. 

Extending this situation, the main part of this paper is devoted to general nonautonomous difference equations in complete metric spaces. Their right-hand side can depend on time allowing to model the dispersal of species in temporally fluctuating environments \cite{bouhours:lewis:16,jacobsen:jin:lewis:15} being not necessarily periodic. Thus, the behaviour depends on both the initial and the actual time. This is why many dynamically relevant objects are contained in the extended state space $\Z\tm X$ (one speaks of nonautonomous sets) \cite{johnson:munoz:09}, rather than being merely subsets of the state space $X$ as in the autonomous case. Furthermore, a complete description of the dynamics in a time-variant setting necessitates a strict distinction between forward and pullback convergence \cite{johnson:munoz:09,kloeden:rasmussen:10}. For this reason only a combination of several attractor notions yields the full picture:
\begin{itemize}
	\item The \emph{pullback attractor} \cite{carvalho:langa:robinson:10,kloeden:00,kloeden:rasmussen:10,poetzsche:10b} is a compact, invariant nonautonomous set which attracts all bounded sets from the past. As fixed target problem, it is based on previous information, at a fixed time from increasingly earlier initial times. Since it consists of bounded entire solutions to a nonautonomous system (see \cite[p.~17, Cor.~1.3.4]{poetzsche:10b}), a pullback attractor can be seen as an extension of the global attractor to nonautonomous problems and apparently captures the essential dynamics to a certain point. Meanwhile the corresponding theory is widely developed in discrete and continuous time. However, pullback attractors reflect the past rather than the future of systems (see \cite{kloeden:poetzsche:rasmussen:11}) and easy examples demonstrate that differential or difference equations with identical pullback attractors might have rather different asymptotics as $t\to\infty$ and possibly feature limit sets, which are not captured by the pullback dynamics. 

	\item This led to the development of \emph{forward attractors}, which are also compact and invariant nonautonomous sets \cite{kloeden:rasmussen:10}. This dual concept depends on information from the future and given a fixed initial time, the actual time increases beyond all bounds --- they are a moving target problem. Forward attractors are not unique, independent of pullback attractors, but often do not exist. Nevertheless, we will describe forward attractors using a pullback construction, even though this has the disadvantage that information on the system over the entire time axis $\Z$ is required. 
	
	\item Therefore, it was suggested in \cite{kloeden:lorenz:16} to work with \emph{forward limit sets}, a concept related to the uniform attractor due to \cite{vishik:92}. They correctly describe the asymptotic behaviour of all forward solutions to a nonautonomous difference equation. These limit sets have forward attraction properties, but different from pullback and forward attractors, they are not (even positively) invariant and constitute a single compact set, rather than a nonautonomous set. Nonetheless, asymptotic forms of positive (and negative) invariance do hold. 
\end{itemize}
The situation for forward attractors and limit sets is not as well-established as their pullback counterparts and deserves to be developed for the above reasons. Their initial construction in \cite{kloeden:lorenz:16,kloeden:yang:16} requires a locally compact state space, but recent continuous-time results in \cite{cui:kloeden:yang:19}, which extend these to infinite-dimensional dynamical systems, will be transferred here. We indeed address nonautonomous difference equations in (not necessarily locally compact) metric and Banach spaces, introduce the mentioned attractor types and study their properties. 

This brings us to our second purpose. The above abstract setting allows concrete applications to a particularly interesting class of infinite-dimensional dynamical systems in discrete time, namely integrodifference equations (IDEs for short). We provide sufficient criteria for the existence of pullback attractors tailor-made for a quite general class of IDEs. Their right-hand sides go beyond pure integral operators and might also include superposition operators, which are used to describe populations having a sedentary fraction. Such results follow from a corresponding theory of set contractions contained in \cite[pp.~15ff]{poetzsche:10b}, \cite[pp.~79ff]{martin:76}. For completely continuous right-hand sides (i.e.,\ Urysohn operators) we construct forward limit sets and provide an application to asymptotically autonomous IDEs. We restrict to rather simple IDEs in the space of continuous functions over a compact domain as state space. More complicated equations and the behaviour of attractors under spatial discretisation will be tackled in future papers. 

The contents of this paper are as follows: In \sref{sec2} we establish the necessary terminology and provide a useful dissipativity condition for nonautonomous difference equations. The key notions related to pullback convergence, i.e.,\ limit sets and attractors are reviewed and established in \sref{sec3}. The subsequent \sref{sec4} addresses the corresponding notions in forward time. In detail, it establishes forward limit sets and their (weakened) invariance properties. For a class of asymptotically autonomous equations it is shown that their forward limit sets coincide with the global attractor of the limit equation. Moreover, a construction of forward attractors is suggested. Finally, in \sref{sec5} we provide some applications to various IDEs. In particular, we illustrate the above theoretical results by studying pullback attractors and forward limit sets. 
\paragraph{Notation}
Let $\R_+:=[0,\infty)$. A \emph{discrete interval} $\I$ is defined as the intersection of a real interval with the integers $\Z$, $\I':=\set{t\in\I:\,t+1\in\I}$ and $\N_0:=\set{0,1,2,\ldots}$. 

On a metric space $(X,d)$, $I_X$ is the identity map, $B_r(x):=\set{y\in X:\,d(x,y)<r}$ the open ball with center $x\in X$ and radius $r>0$, and $\bar B_r(0)$ denotes its closure. We write 
$
	\dist{x,A}:=\inf_{a\in A}d(x,a)
$
for the distance of $x$ from a set $A\subseteq X$ and
$
	B_r(A):=\set{x\in X:\,\dist{x,A}<r}
$
for its $r$-neighbourhood. The \emph{Hausdorff semidistance} of bounded and closed subsets $A,B\subseteq X$ is defined as
$$
	\dist{A,B}:=\sup_{a\in A}\inf_{b\in B}d(a,b).
$$
The \emph{Kuratowski measure of noncompactness} on $X$ (cf.\ \cite[pp.~16ff, I.5]{martin:76}) is denoted by $\chi:{\mathfrak B}(X)\to\R$, where ${\mathfrak B}(X)$ stands for the family of bounded subsets of $X$. 

A mapping $\eF:X\to X$ is said to be \emph{bounded}, if it maps bounded subsets of $X$ into bounded sets and \emph{globally bounded}, if $\eF(X)$ is bounded. We say a bounded $\eF$ satisfies a \emph{Darbo condition}, if there exists a real constant $k\geq 0$ such that
$$
	\chi\intoo{\eF(B)}\leq k\chi(B)\fall B\in{\mathfrak B}(X).
$$
The smallest such $k$ is the \emph{Darbo constant} $\dar{\eF}\in[0,\infty)$ of $\eF$. A \emph{completely continuous} mapping $\eF$ is bounded, continuous and satisfies $\dar{\eF}=0$. 

A subset $\cA\subseteq\I\tm X$ with $t$-\emph{fibres} $\cA(t):=\{x\in X:\,(t,x)\in\cA\}$, $t\in\I$, is called \emph{nonautonomous set}. If all fibres $\cA(t)\subseteq X$, $t\in\I$, are compact, then $\cA$ is denoted as \emph{compact} nonautonomous set and we proceed accordingly with other topological notions. Furthermore, one speaks of a \emph{bounded} nonautonomous set $\cA$, if there exists real $R>0$ and a point $x_0\in X$ such that $\cA(t)\subseteq B_R(x_0)$ holds for all $t\in\I$.

Finally, on a Banach space $X$, $L(X)$ denotes the space of bounded linear operators and $\rho(\eL)$ is the \emph{spectral radius} of a $\eL\in L(X)$. 
\section{Nonautonomous difference equations}
\label{sec2}
Unless otherwise noted, let $(X,d)$ be a complete metric space. We consider nonautonomous difference equations in the abstract form
\begin{equation}
	\tag{$\Delta$}
	\boxed{u_{t+1}=\eF_t(u_t)}
	\label{deq}
\end{equation}
with continuous right-hand sides $\eF_t:U_t\to X$ and defined on closed sets $U_t\subseteq X$, $t\in\I'$. For an \emph{initial time} $\tau\in\I$, a \emph{forward solution} to \eqref{deq} is a sequence $(\phi_t)_{\tau\leq t}$ with $\phi_t\in U_t$ satisfying 
\begin{equation}
	\phi_{t+1}\equiv\eF_t(\phi_t)
	\label{solid}
\end{equation}
for all $\tau\leq t$, $t\in\I'$, while an \emph{entire solution} $(\phi_t)_{t\in\I}$ satisfies \eqref{solid} on $\I'$. The unique forward solution starting at $\tau\in\I$ in $u_\tau\in U_\tau$ is denoted by $\vphi(\cdot;\tau,u_\tau)$; it is denoted as \emph{general solution} to \eqref{deq} and reads as
\begin{equation}
	\vphi(t;\tau,u_\tau)
	:=
	\begin{cases}
		\eF_{t-1}\circ\ldots\circ\eF_\tau(u_\tau),&\tau<t,\\
		u_\tau,&\tau=t,
	\end{cases}
	\label{process}
\end{equation}
as long as the compositions stay in $U_t$. Under the inclusion $\eF_t(U_t)\subseteq U_{t+1}$, $t\in\I'$, the general solution $\vphi(t;\tau,\cdot):U_\tau\to U_t$ exists for all $\tau\leq t$ and the \emph{process property}
\begin{equation}
	\vphi(t;s,\vphi(s;\tau,u))=\vphi(t;\tau,u)\fall\tau\leq s\leq t,\,u\in U_\tau
	\label{semigroup}
\end{equation}
holds; we introduce the nonautonomous set $\cU:=\set{(t,u)\in\I\tm X:\,u\in U_t}$. 

One denotes \eqref{deq} as $\theta$-\emph{periodic} with some $\theta\in\N$, if $\eF_t=\eF_{t+\theta}$, $U_t=U_{t+\theta}$ and tacitly $\I=\Z$ hold for all $t\in\Z$. In this case the general solution satisfies
\begin{equation}
	\vphi(t+\theta;\tau+\theta,u_\tau)=\vphi(t;\tau,u_\tau)\fall\tau\leq t,\,(\tau,u_\tau)\in\cU
	\label{noper}
\end{equation}
yielding a rather tame time-dependence. An \emph{autonomous} equation \eqref{deq} is $1$-periodic.

A nonautonomous set $\cA\subseteq\cU$ is called \emph{positively} or \emph{negatively invariant} (w.r.t.\ the difference equation \eqref{deq}), if the respective inclusion
\begin{align*}
	\eF_t\intoo{\cA(t)}&\subseteq\cA(t+1),&
	\cA(t+1)&\subseteq\eF_t\intoo{\cA(t)}\fall t\in\I'
\end{align*}
holds; an \emph{invariant} set $\cA$ is both positively and negatively invariant, that is, $\eF_t\intoo{\cA(t)}=\cA(t+1)$ for all $t\in\I'$. One denotes $\cA$ as $\theta$-\emph{periodic}, if $\cA(t)=\cA(t+\theta)$ holds for all $t\in\I$ with $t+\theta\in\I$. 

The next two subsections provide some preparations on nonautonomous difference equations in Banach spaces $(X,\norm{\cdot})$: 
\subsection{Semilinear difference equations}
Let $\eL_t\in L(X)$, $t\in\I'$, be a sequence of bounded linear operators. For a \emph{linear difference equation}
$$
	\boxed{u_{t+1}=\eL_tu_t}
$$
we define the \emph{transition operator} $\Phi:\set{(t,\tau)\in\I^2:\,\tau\leq t}\to L(X)$ by
$$
	\Phi(t,\tau)
	:=
	\begin{cases}
		\eL_{t-1}\cdots\eL_\tau,&\tau<t,\\
		I_X,&t=\tau.
	\end{cases}
$$
Then \eqref{deq} is understood as \emph{semilinear}, if its right-hand side can be represented as 
\begin{equation}
	\eF_t=\eL_t+\eN_t
	\label{rhssemi}
\end{equation}
with continuous mappings $\eN_t:U_t\to X$, $t\in\I'$. The variation of constants formula \cite[p.~100, Thm.~3.1.16]{poetzsche:10b} yields the general solution of \eqref{deq} in the form
\begin{equation}
	\vphi(t;\tau,u_\tau)
	=
	\Phi(t,\tau)u_\tau+\sum_{s=\tau}^{t-1}\Phi(t,s+1)\eN_s\intoo{\vphi(s;\tau,u_\tau)}
	\fall\tau\leq t,\,u_\tau\in U_\tau.
	\label{voc}
\end{equation}

The following result will be helpful in the construction of absorbing sets: 
\begin{lemma}\label{lemabs}
	Let $\eF_t:U_t\to U_{t+1}$ be of semilinear form \eqref{rhssemi} and suppose there exist reals $\alpha_t\geq 0$, $K\geq 1$ with
	\begin{equation}
		\norm{\Phi(t,s)}\leq K\prod_{r=s}^{t-1}\alpha_r\fall s\leq t. 
		\label{lemabs1}
	\end{equation}
	If there exist reals $a_t\geq 0$, $b_t\geq 0$ such that the nonlinearity fulfills 
	\begin{equation}
		\norm{\eN_t(u)}\leq b_t+a_t\norm{u}\fall t\in\I',\,u\in U_t,
		\label{lemabs2}
	\end{equation}
	then the general solution of \eqref{deq} satisfies the estimate
	\begin{equation}
		\norm{\vphi(t;\tau,u_\tau)}
		\leq
		K\norm{u_\tau}\prod_{r=\tau}^{t-1}(\alpha_r+Ka_r)+K\sum_{s=\tau}^{t-1}b_s\prod_{r=s+1}^{t-1}(\alpha_r+Ka_r)
		\label{lemabs3}
	\end{equation}
	for all $\tau\leq t$ and $u_\tau\in U_\tau$. 
\end{lemma}
\begin{remark}[linear growth]
	In case $\eL_t\equiv 0$ on $\I'$ one can choose $K=1$, $\alpha_t=0$ in \eqref{lemabs1} and the estimate \eqref{lemabs3} simplifies to
	$$
		\norm{\vphi(t;\tau,u_\tau)}
		\leq
		\norm{u_\tau}\prod_{r=\tau}^{t-1}a_r+\sum_{s=\tau}^{t-1}b_s\prod_{r=s+1}^{t-1}a_r
		\fall\tau\leq t,\,u_\tau\in U_\tau.
	$$
\end{remark}
\begin{proof}
	Let $\tau\in\I$. It is convenient to abbreviate $e_\alpha(t,s):=\prod_{r=s}^{t-1}\alpha_r$ and we first assume that $\alpha_t\neq 0$, $t\in\I'$. Given $u_\tau\in U_\tau$, from \eqref{voc} and \eqref{lemabs1} we obtain
	\begin{eqnarray*}
		\norm{\vphi(t;\tau,u_\tau)}
		& \leq &
		Ke_\alpha(t,\tau)\norm{u_\tau}
		+
		K\sum_{s=\tau}^{t-1}e_\alpha(t,s+1)\norm{\eN_s(\vphi(s;\tau,u_\tau))}\\
		& \stackrel{\eqref{lemabs2}}{\leq} &
		Ke_\alpha(t,\tau)\norm{u_\tau}
		+
		K\sum_{s=\tau}^{t-1}e_\alpha(t,s+1)\intoo{b_s+a_s\norm{\vphi(s;\tau,u_\tau)}}
	\end{eqnarray*}
	and therefore the sequence $u(t):=\norm{\vphi(t;\tau,u_\tau)}e_\alpha(\tau,t)$ satisfies
	$$
		u(t)
		\leq
		K\norm{u_\tau}
		+
		K\sum_{s=\tau}^{t-1}b_se_\alpha(\tau,s+1)
		+
		K\sum_{s=\tau}^{t-1}\frac{a_s}{\alpha_s}u(s)\fall\tau\leq t.
	$$
	Thus, the Gr\"onwall inequality from \cite[p.~348, Prop.~A.2.1(a)]{poetzsche:10b} implies
	$$
		u(t)
		\leq
		Ke_{1+\frac{Ka}{\alpha}}(t,\tau)\norm{u_\tau}
		+
		K\sum_{s=\tau}^{t-1}b_se_\alpha(\tau,s+1)e_{1+\frac{Ka}{\alpha}}(t,s+1)
	$$
	and consequently
	$$
		\norm{\vphi(t;\tau,u_\tau)}
		\leq
		Ke_{\alpha+Ka}(t,\tau)\norm{u_\tau}+K\sum_{s=\tau}^{t-1}b_se_{\alpha+Ka}(t,s+1)
		\fall\tau\leq t,
	$$
	which is the claimed inequality \eqref{lemabs3}. 
	\qed
\end{proof}
\subsection{Additive difference equations}
We now address right-hand sides
\begin{equation}
	\eF_t=\eG_t+\eK_t, 
	\label{rhsadd}
\end{equation}
where $\eG_t:U_t\to X$ is bounded and continuous, while $\eK_t:U_t\to X$, $t\in\I'$, is assumed to be completely continuous. 
\begin{lemma}\label{lem22}
	If $\eF_t:U_t\to U_{t+1}$ is of additive form \eqref{rhsadd}, then the general solution of \eqref{deq} satisfies
	$$
		\dar{\vphi(t;\tau,\cdot)} \leq\prod_{s=\tau}^{t-1}\dar{\eG_s}\fall \tau\leq t.
	$$
\end{lemma}
\begin{proof}
	Since $\eF_t=\eG_t+\eK_t:U_t\to X$ for every $t\in\I'$ is continuous and bounded, their composition \eqref{process} is also continuous and bounded. The estimate for the Darbo constant of $\vphi(t;\tau,\cdot)$ will be established by mathematical induction. For $t=\tau$ the assertion is clear, since $\vphi(\tau;\tau,\cdot)=I_X$ and the Lipschitz constant of the identity mapping is $1$; it provides an upper bound for the Darbo constant (see \cite[p.~81, Prop.~5.3]{martin:76}). For times $t\geq\tau$, from $\dar{\eK_t\circ\vphi(t;\tau,\cdot)}=0$, which holds because $\eK_t$ is completely continuous (cf.~\cite[p.~82, Prop.~5.4]{martin:76}), it follows that 
	\begin{align*}
		\dar{\vphi(t+1;\tau,\cdot)}
		& =
		\dar{\eG_t\circ\vphi(t;\tau,\cdot)+\eK_t\circ\vphi(t;\tau,\cdot)} 
		\leq \dar{\eG_t\circ\vphi(t;\tau,\cdot)}\\
		&\leq \dar{\eG_t}\dar{\vphi(t;\tau,\cdot)}
		=
		\prod_{s=\tau}^t\dar{\eG_s}\fall\tau\leq t
	\end{align*}
	from \cite[pp.~79--80, Prop.~5.1]{martin:76}. This establishes the claim. 
	\qed
\end{proof}
\section{Pullback convergence}
\label{sec3}
In this section, suppose that $\I$ is unbounded below and that $\eF_t:U_t\to U_{t+1}$, $t\in\I'$, i.e.,\ \eqref{deq} generates a process on $\cU$. 

A difference equation \eqref{deq} is said to be \emph{pullback asymptotically compact}, if for every $\tau\in\I$, every sequence $(s_n)_{n\in\N}$ in $\N_0$ with $\lim_{n\to\infty}s_n=\infty$ and every bounded sequence $(a_n)_{n\in\N}$ with $a_n\in\cU(\tau-s_n)$, the sequence $\intoo{\vphi(\tau;\tau-s_n,a_n)}_{n\in\N}$ possesses a convergent subsequence. 
\subsection{Pullback limit sets}
The \emph{pullback limit set} $\omega_{\cA}\subseteq\cU$ of a bounded subset $\cA\subseteq\cU$ is given by the fibres
\begin{equation}
	\omega_{\cA}(\tau):=\bigcap_{0\leq s}\overline{\bigcup_{s\leq t}\vphi(\tau;\tau-t,\cA(\tau-t))}
	\fall\tau\in\I.
	\label{pb01}
\end{equation}
For pullback asymptotically compact nonautonomous difference equations \eqref{deq} it is shown in \cite[p.~14, Thm.~1.2.25]{poetzsche:10b} that $\omega_{\cA}$ is nonempty, compact, invariant and \emph{pullback attracts} $\cA$, i.e.,\ the limit relation
\begin{equation}
	\lim_{s\to\infty}\dist{\vphi(\tau;\tau-s,\cA(\tau-s)),\omega_{\cA}(\tau)}=0\fall\tau\in\I
	\label{pb02}
\end{equation}
holds. For positively invariant sets $\cA$ the defining relation \eqref{pb01} simplifies to
\begin{equation}
	\omega_{\cA}(\tau)=\bigcap_{0\leq s}\overline{\vphi(\tau;\tau-s,\cA(\tau-s))}. 
	\label{pb03}
\end{equation}
Therefore, as a fundamental tool for the construction of pullback limit sets and attractors, as well as for forward attractors in \sref{sec43}, we state
\begin{proposition}\label{prop31}
	Suppose that \eqref{deq} has a nonempty, positively invariant, closed and bounded subset $\cA\subseteq\cU$. If \eqref{deq} is pullback asymptotically compact, then the fibres
	\begin{equation}
		\cA^\star(\tau):=\bigcap_{0\leq s}\vphi(\tau;\tau-s,\cA(\tau-s))
		\fall\tau\in\I
		\label{prop31a}
	\end{equation}
	define a maximal invariant, nonempty and compact nonautonomous set $\cA^\star\subseteq\cA$, which pullback attracts $\cA$. 
\end{proposition}
If the nonautonomous set $\cA$ is even compact, then \pref{prop31} applies without the asymptotic compactness assumption. 
\begin{proof} % \dontknow Thm 5.1 from Cui, Kloeden, Yang - proved by Huy
	Since \eqref{deq} generates a continuous process $\vphi$ in discrete time, the assertion results via an adaption of \cite[Prop.~5]{kloeden:marin:10}, where pullback asymptotic compactness yields that the intersection of the nested sets in \eqref{prop31a} is nonempty. 
	\qed
\end{proof}
\subsection{Pullback attractors}
\label{sec32}
A \emph{pullback attractor} $\cA^\ast\subseteq\cU$ of \eqref{deq} is a nonempty, compact, invariant nonautonomous set which pullback attracts all bounded nonautonomous sets $\cB\subseteq\cU$. Bounded pullback attractors are unique and allow the dynamical characterisation
$$
	\cA^\ast
	=
	\set{(\tau,u)\in\cU\,
	\left|
	\begin{array}{l}
	\text{there exists a bounded entire solution}\\
	(\phi_t)_{t\in\I}\text{ of \eqref{deq} satisfying }\phi_\tau=u
	\end{array}
	\right.}
$$
(cf.~\cite[p.~17, Cor.~1.3.4]{poetzsche:10b}). Despite being pullback attracting nonautonomous sets within $\cA$, the set $\cA^\star$ constructed in \pref{prop31} needs not to be a pullback attractor, since nothing was assumed outside of $\cA$. Remedy provides the notion of a \emph{pullback dissipative} difference equation \eqref{deq}. This means there exists a bounded set $\cA\subseteq\cU$ such that for every $\tau\in\I$ and every bounded nonautonomous set $\cB\subseteq\cU$ there is an \emph{absorption time} $S=S(\tau,\cB)\in\N$ such that 
$$
	\vphi(\tau;\tau-s,\cB(\tau-s))\subseteq\cA(\tau)\fall s\geq S.
$$
For a \emph{uniformly pullback dissipative} equation \eqref{deq} the absorption time $S$ is independent of $\tau$. One denotes $\cA$ as a \emph{pullback absorbing set}. 

If $\cA$ is pullback absorbing, then the set $\cA^\star$ obtained from \pref{prop31} becomes a pullback attractor, i.e.,\ $\cA^\star=\cA^\ast$, and one has the characterisation
\begin{equation}
	\cA^\ast=\omega_{\cA}.
	\label{charpb}
\end{equation}
A possibility to construct pullback absorbing sets provides
\begin{proposition}[Pullback absorbing set]\label{proppbab}
	On a Banach space $X$, let $\rho>0$ and $\eF_t:U_t\to U_{t+1}$ be of semilinear form \eqref{rhssemi} satisfying \eqref{lemabs1}, \eqref{lemabs2}. If the limit relations
	\begin{align*}
		\lim_{s\to\infty}\prod_{r=\tau-s}^{\tau-1}(\alpha_r+Ka_r)&=0,&
		R_\tau&:=K\sum_{s=-\infty}^{\tau-1}b_s\prod_{r=s+1}^{\tau-1}(\alpha_r+Ka_r)<\infty
	\end{align*}
	hold for all $\tau\in\I$, then the difference equation \eqref{deq} is pullback dissipative with absorbing set $\cA:=\set{(\tau,u)\in\cU:\,\norm{u}\leq\rho+R_\tau}$. In case $\lim_{s\to\infty}\sup_{\tau\in\I}\prod_{r=\tau-s}^{\tau-1}(\alpha_r+Ka_r)=0$ holds, the difference equation \eqref{deq} is uniformly pullback dissipative. 
\end{proposition}
\begin{proof}
	The assertion follows from \lref{lemabs} by passing over to the pullback limit $\tau\to-\infty$ in the estimate \eqref{lemabs3}. 
	\qed
\end{proof}

A construction of pullback attractors $\cA^\ast$ based on set contractions, rather than asymptotic compactness, is suitable for later applications to integrodifference equations (see \sref{sec5}): 
\begin{theorem}\label{thmpullback}
	If a difference equation \eqref{deq} of additive form \eqref{rhsadd} is uniformly pullback dissipative and 
	$$
		\prod_{s=-\infty}^{T-1}\dar{\eG_s}=0\quad\text{for some }T\in\I
	$$
	holds, then there exists a unique bounded pullback attractor of \eqref{deq}. 
\end{theorem}
\begin{remark}[periodic equations]
	For $\theta$-periodic difference equations \eqref{deq} and sets $\cA$, it results from \eqref{noper} that also the pullback limit sets $\omega_\cA$ from \eqref{pb01}, the set $\cA^\star$ from \pref{prop31} and the pullback attractor $\cA^\ast$ are $\theta$-periodic (cf.\ \cite[pp.~21ff, Sect.~1.4]{poetzsche:10b}). Furthermore, \tref{thmpullback} applies when $\prod_{s=0}^{\theta-1}\dar{\eG_s}<1$. 
\end{remark}
\begin{proof} 
	The terminology of \cite{poetzsche:10b} and results therein will be used. Let $\hat\cB$ denote the family of all bounded sets in $\I\tm X$. Then \lref{lem22} ensures that the general solution $\vphi(t;\tau,\cdot)$ is $\hat\cB$-con\-trac\-ting in the sense of \cite[p.~15, Def.~1.2.26(i)]{poetzsche:10b}. 

Since \eqref{deq} has a bounded absorbing set $\cA$, for every bounded nonautonomous set $\cB$, there exists an $S\in\N_0$ such that $\vphi(\tau;\tau-s,\cB(\tau-s))\subseteq\cA(t)$ holds for all $s\geq S$. This implies that the $S$-truncated orbit $\gamma_\cB^S$, fibrewise given by
	$$
		\gamma_\cB^S(\tau)
		:=
		\bigcup_{s\geq S}\vphi(\tau;\tau-s,\cB(\tau-s))\subseteq\cA(\tau)\fall\tau\in\I, 
	$$
	is bounded. Hence, \cite[p.~16, Prop.~1.2.30]{poetzsche:10b} implies that \eqref{deq} is $\hat\cB$-asymptotically compact, so \eqref{deq} has a pullback attractor $\cA^\ast$ by \cite[p.~19, Thm.~1.3.9]{poetzsche:10b}. 
	
	Finally, the pullback attractor $\cA^\ast$ is contained in the closure of the absorbing set $\cA$, so is bounded and thus uniquely determined. 
	\qed
\end{proof}
\section{Forward convergence}
\label{sec4}
In the previous section, we constructed pullback attractors of pullback asymptotically compact nonautonomous difference equations \eqref{deq} as pullback limit sets of such absorbing sets. Our next aim is to provide related notions in forward time. Due to the conceptional difference between pullback and forward convergence some modifications are necessary, yet. 

Above all, this requires a discrete interval $\I$ to be unbounded above. Now the right-hand sides $\eF_t:U_t\to U_{t+1}$, $t\in\I$, are defined on a common closed subset $U_t=U\subseteq X$, i.e.,\ the extended state space $\cU=\I\tm U$ has constant fibres. Therefore, the general solution $\vphi:\set{(t,\tau,u)\in\I\tm\cU:\,\tau\leq t}\to U$ is well-defined.

Given a nonautonomous set $\cA\subseteq\cU$, a difference equation \eqref{deq} is said to be 
\begin{itemize}
	\item \emph{$\cA$-asymptotically compact}, if there exists a compact set $K\subseteq U$ such that $K$ forward attracts $\cA(\tau)$, i.e.,\
	\begin{align*}
		\lim_{s\to\infty}\dist{\vphi(\tau+s;\tau,\cA(\tau)),K}=0 \fall\tau\in\I, 
	\end{align*}
	
	\item \emph{strongly $\cA$-asymptotically compact}, if there exists a compact set $K\subseteq U$ so that every sequence $\intoo{(s_n,\tau_n)}_{n\in\N}$ in $\N\tm\I$ with $s_n\to\infty$, $\tau_n\to\infty$ as $n\to\infty$ yields
	\begin{align*}
		\lim_{n\to\infty}\dist{\vphi(\tau_n+s_n;\tau_n,\cA(\tau_n)), K}=0. 
	\end{align*}
\end{itemize}
\begin{remark}
	If $\cA$ is positively invariant, then strong $\cA$-asymptotic compactness (needed in \tref{thm410} below) is a tightening of $\cA$-asymptotic compactness (required in \tref{thm49}). Indeed, suppose that the sequence $(\dist{\vphi(t_n;\tau,\cA(\tau)),K})_{n\in\N}$ does not converge to $0$. Hence, the strong $\cA$-asymptotic compactness of \eqref{deq} and positive invariance of $\cA$ yields the contradiction
	\begin{align*}
		\dist{\vphi(t_n;\tau,\cA(\tau)),K}
		&=\dist{\eF_{t_n}(\vphi(t_n-1;\tau,\cA(\tau)),K}\\
		&\leq
		\dist{\eF_{t_n}(\cA(t_n-1)),K}
		\xrightarrow[n\to\infty]{}
		0.
	\end{align*}
\end{remark}
\subsection{Forward limit sets}
Let us investigate the forward dynamics of \eqref{deq} inside a nonautonomous set $\cA$. We first capture the forward limit points from a single fibre $\cA(\tau)$: 
\begin{lemma}\label{lem42}
	Suppose that $\cA\neq\emptyset$ is a bounded nonautonomous set. If \eqref{deq} is $\cA$-asymp\-totically compact with a compact subset $K\subseteq U$, then the fibres
	\begin{align}
	\Omega_\cA(\tau):=\bigcap_{0\leq s}\overline{\bigcup_{s\leq t}\vphi(\tau+t;\tau,\cA(\tau))}\subseteq K
	\fall\tau\in\I
	\label{lem42a}
	\end{align}
	are nonempty, compact, and forward attract $\cA(\tau)$, i.e.,\ 
	\begin{equation}
		\lim_{s\to\infty}\dist{\vphi(\tau+s;\tau,\cA(\tau)),\Omega_\cA(\tau)}=0.
		\label{lem42b}
	\end{equation} 
\end{lemma}
An analogous result for pullback limit sets is given in \cite[p.~9, Lemma~1.2.12]{poetzsche:10b}. 
\begin{remark}[characterisation of $\Omega_{\cA}(\tau)$]\label{rem43}
	The fibres $\Omega_\cA(\tau)$, $\tau\in\I$, consist of points $v$ such that there is a sequence $\intoo{(s_n,a_n)}_{n\in\N}$ with $\lim_{n\to\infty}s_n=\infty$, $a_n\in\cA(\tau)$ and
	\begin{equation}
		\lim_{n\to\infty}\vphi(\tau+s_n;\tau,a_n)=v.
		\label{rem43a}
	\end{equation}
	This readily yields the monotonicity $\cA_1\subseteq\cA_2$ $\Rightarrow$ $\Omega_{\cA_1}(\tau)\subseteq\Omega_{\cA_2}(\tau)$ for all $\tau\in\I$. 
\end{remark}
\begin{proof} % \dontknow Lemma 2.2 in \cite{cui:kloeden:yang:19} - proved by Huy
	Let $\tau\in\I$. Given a sequence $y_n:=\vphi(\tau+s_n;\tau,a_n)\in\vphi(\tau+s_n;\tau,\cA(\tau))$ with $s_n\xrightarrow[n\to\infty]{}\infty$ and $a_n\in\cA(\tau)$, by the $\cA$-asymptotic compactness of \eqref{deq}, we obtain
	\begin{align*}
		0\leq\dist{y_n,K}\leq\dist{\vphi(\tau+s_n,\tau,\cA(\tau)),K}\xrightarrow[n\to\infty]{}0.
	\end{align*}
	Since $K\subseteq U$ is compact, $\dist{y_n,K}=\min_{k\in K}d(y_n,k)$. This implies that there exist a sequence $(k_n)_{n\in\N}$ in $K$ satisfying
	\begin{align*}
		d(y_n,k_n) = \dist{y_n,K}\xrightarrow[n\to\infty]{}0,
	\end{align*}
	and a subsequence $(k_{n_j})_{j\in\N}$ converging to $\bar{k}\in K$. Thus,
	\begin{align*}
		0\leq d\intoo{y_{n_j},\bar{k}}
		\leq d\intoo{y_{n_j},k_{n_j}}+d\intoo{k_{n_j},\bar{k}}
		=\dist{y_{n_j},K}+d\intoo{k_{n_j},\bar{k}}\xrightarrow[j\to\infty]{}0,
	\end{align*}
	which implies that the subsequence $(y_{n_j})_{j\in\N}$ converges to $\bar{k}$. Hence, by the characterisation \eqref{rem43a}, $\bar{k}\in\Omega_\cA(\tau)$, i.e.,\ $\Omega_\cA(\tau)$ is nonempty.

	Now choose a sequence $(v_n)_{n\in\N}$ in $\Omega_\cA(\tau)$. By \rref{rem43}, for each fixed $n\in\N$, there is a sequence $\intoo{(s_m^n,a_m^n)}_{m\in\N}$ satisfying $\lim_{m\to\infty}s_m^n=\infty$ and $a_m^n\in\cA(\tau)$ for all $m\in\N$ such that
	\begin{align*}
		\lim_{m\to\infty}\vphi(\tau+s_m^n;\tau,a_m^n)=v_n,
	\end{align*}
	i.e.,\ for every $\eps>0$, there is a $M=M(n,\eps)\in\N$ such that
	\begin{align*}
		d\intoo{\vphi(\tau+s_m^n;\tau,a_m^n),v_n}<\eps\fall m\geq M.
	\end{align*}
	Since $\lim_{m\to\infty}s_m^n=\infty$, there is a $M'=M'(n)\in\N$ satisfying $s_m^n>n$ for all $m\geq M'$. Pick $\eps=\frac{1}{n}$, $(\bar s_n)_{n\in\N}$ as a subsequence $(s_{m_n}^n)_{n\in\N}$ of $(s_m^n)_{n\in\N}$ and $(\bar a_n)_{n\in\N}$ as a subsequence $(a_{m_n}^n)_{n\in\N}$ of $(a_m^n)_{n\in\N}$ in such a way that
	\begin{align*}
		m_1&=\max\set{1,M,M'}+1,&
		m_{n+1}&=\max\set{1,M,M',m_n}+1 \fall n\in\N.
	\end{align*}
	Clearly, $m_{n+1}>\max\set{1,M,M',m_n}$ for $n\in\N$. Hence, we constructed a sequence $\intoo{(\bar s_n,\bar a_n)}_{n\in\N}$ with $\lim_{n\to\infty}\bar s_n=\infty$ and $\bar a_n\in\cA(\tau)$ such that
	\begin{align*}
		d\intoo{\vphi(\tau+\bar s_n;\tau,\bar a_n),v_n}<\tfrac{1}{n} \fall n\in\N.
	\end{align*}
	Therefore,
	\begin{align*}
		0\leq\dist{v_n,K} &\leq d\intoo{\vphi(\tau+\bar s_n;\tau,\bar a_n),v_n}+\dist{\vphi(\tau+\bar s_n;\tau,\cA(\tau)),K}\\
		&\leq \tfrac{1}{n}+\dist{\vphi(\tau+\bar s_n;\tau,\cA(\tau)),K} \xrightarrow[n\to\infty]{}0.
	\end{align*}
	Similarly as above, because $K$ is compact, there is a subsequence $(v_{n_j})_{j\in\N}$ converging to $\bar{v}\in K$. Moreover, since $\Omega_\cA(\tau)$ is closed by definition, $\bar{v}\in\Omega_\cA(\tau)$, which implies that $\Omega_\cA(\tau)$ is compact. Also note that $\lim_{n\to\infty}\dist{v_n,K}=0$, so $v_n\in K$, i.e.,\ $\Omega_\cA(\tau)\subseteq K$.
	
	Suppose that $\Omega_\cA(\tilde{\tau})$ does not forward attract $\cA(\tilde{\tau})$ for some $\tilde{\tau}\in\I$, i.e.,\ there exist a real $\tilde{\eps}>0$ and a sequence $(\tilde{s}_n)_{n\in\N}$ in $\N$ with $\lim_{n\to\infty}\tilde{s}_n=\infty$ and
	\begin{align}
		\dist{\vphi(\tilde{\tau}+\tilde{s}_n;\tilde{\tau},\cA(\tilde{\tau})),\Omega_\cA(\tilde{\tau})}\geq \tilde{\eps} \fall n\in\N.
		\label{star1}
	\end{align}
	Although the supremum in the Hausdorff semidistance in the left-hand side of \eqref{star1} may not be attained due to no condition ensuring that the image $\vphi(\tilde{\tau}+s_n;\tilde{\tau},\cA(\tilde{\tau}))$ is compact, there still exists a point $\tilde{y}_n:=\vphi(\tilde{\tau}+\tilde{s}_n;\tilde{\tau},\tilde{a}_n) \in\vphi(\tilde{\tau}+\tilde{s}_n;\tilde{\tau},\cA(\tilde{\tau}))$ for each $n\in\N$ with $\tilde{a}_n\in\cA(\tilde{\tau})$ such that
	\begin{align*}
		\dist{\vphi(\tilde{\tau}+\tilde{s}_n;\tilde{\tau},\cA(\tilde{\tau})),\Omega_\cA(\tilde{\tau})}-\tfrac{\tilde{\eps}}{2} &\leq \dist{\tilde{y}_n,\Omega_\cA(\tilde{\tau})}\\
		&\leq \dist{\vphi(\tilde{\tau}+\tilde{s}_n;\tilde{\tau},\cA(\tilde{\tau})},\Omega_\cA(\tilde{\tau})).
	\end{align*}
	The above inequalities in fact give $\dist{\tilde{y}_n,\Omega_\cA(\tilde{\tau})}\geq\tfrac{\tilde{\eps}}{2}$ for all $n\in\N$. On the other hand, since $\tilde{y}_n\in\vphi(\tilde{\tau}+s_n;\tilde{\tau},\cA(\tilde{\tau}))$, we obtain
	\begin{align*}
		\dist{\tilde{y}_n,K}\leq \dist{\vphi(\tilde{\tau}+\tilde{s}_n;\tilde{\tau},\cA(\tilde{\tau})),K}\xrightarrow[n\to\infty]{}0
	\end{align*}
	and thus $\tilde{y}_n\in K$. Moreover, since $K$ is compact, there is a convergent subsequence $(\tilde{y}_{n_j})_{j\in\N}$ with limit $\tilde{y}\in K$. This shows $\tilde{y}\in\Omega_\cA(\tilde{\tau})$ by definition, and thus
	\begin{align*}
		\dist{\tilde{y},\Omega_\cA(\tilde{\tau})}<\eps \fall\eps>0,
	\end{align*}
	a contradiction to \eqref{star1}. Hence, every $\Omega_\cA(\tau)$ must forward attract $\cA(\tau)$. 
	\qed
\end{proof}

\begin{corollary}\label{cor44}
	If in addition $\cA$ is positively invariant w.r.t.\ \eqref{deq}, then the inclusions $\Omega_\cA(\tau)\subseteq\Omega_\cA(\tau+1)$ hold for all $\tau\in\I$. 
\end{corollary}
Owing to the positive invariance of $\cA$, every $\Omega_\cA(\tau)$ can also be written as
\begin{align}
	\Omega_\cA(\tau) = \bigcap_{0\leq s} \overline{\vphi(\tau+s;\tau,\cA(\tau))}\fall\tau\in\I.
	\label{lem42c}
\end{align}
Comparing the respective relations \eqref{lem42a} and \eqref{pb01}, \eqref{lem42b} and \eqref{pb02}, \eqref{lem42c} and \eqref{pb03} shows that the fibres $\Omega_{\cA}(\tau)$ are counterparts to the pullback limit set $\omega_{\cA}$. However, their invariance property is missing. This is easily demonstrated by
\begin{example}\label{exninv}
	Let $\I=\N_0$ and $\alpha\in(0,\tfrac{1}{2}]$. The difference equation $u_{t+1}=\alpha u_t+\alpha^t$ in $\R$ possesses the positively invariant and bounded set $\cA=\N_0\tm[-\tfrac{1}{\alpha},\tfrac{1}{\alpha}]$. This yields the apparently not even positively invariant sets $\Omega_\cA(\tau)=\set{0}$ for all $\tau\in\N_0$. 
\end{example}
\begin{proof} % \dontknow Lemma 2.2 in \cite{cui:kloeden:yang:19} - proved by Huy
	Let $\tau\in\I$. Given a point $v\in\Omega_\cA(\tau)$, thanks to $\eF_\tau\intoo{\cA(\tau)}\subseteq\cA(\tau+1)$ and the process property \eqref{semigroup}, we obtain 
	\begin{align*}
		v\in\vphi(\tau+s;\tau,\cA(\tau))\subseteq\vphi(\tau+s;\tau+1,\cA(\tau+1))\subseteq U \fall s>0.
	\end{align*}
	This implies $\Omega_\cA(\tau)\subseteq\Omega_\cA(\tau+1)\subseteq U$.
	\qed
\end{proof}

While the fibres $\Omega_{\cA}(\tau)$ from \lref{lem42} yield the long term behaviour starting from a single fibre of $\cA$, the following result addresses all forward limit sets of \eqref{deq} originating from within an entire nonautonomous set $\cA$. 
\begin{theorem}[Forward $\omega$-limit sets]\label{thm46}
	Suppose that $\cA\neq\emptyset$ is positively invariant and bounded. If \eqref{deq} is $\cA$-asymptotically compact with a compact subset $K\subseteq U$, then 		
	\begin{align*}
		\omega_{\cA}^-&:=\bigcap_{\tau\in\I}\Omega_\cA(\tau),&
		\ocA&:=\overline{\bigcup_{\tau\in\I}\Omega_\cA(\tau)}\subseteq K
	\end{align*}
	are nonempty and compact. In particular, $\omega_{\cA}^+$ forward attracts $\cA$, i.e.,\ 
	\begin{align*}
		\lim_{s\to\infty}\dist{\varphi(\tau+s;\tau,\cA(\tau)),\omega_{\cA}^+}&=0
		\fall\tau\in\I
	\end{align*}
	and called \emph{forward $\omega$-limit set} of $\cA$. 
\end{theorem}
Due to \cref{cor44}, $\omega_{\cA}^+$ is a union over nondecreasing sets and actually a limit. 
\begin{remark}[characterisation of $\omega_{\cA}^+$]\label{rem45}
	The forward $\omega$-limit set $\ocA$ consists of all points $v$ such that there is a sequence $\intoo{(s_n,\tau_n,a_n)}_{n\in\N}$ with $\lim_{n\to\infty}\tau_n=\infty$, $(\tau_n,a_n)\in\cA$ and $s_n\in\N_0$ satisfying
	\begin{align*}
		\lim_{n\to\infty}\vphi(\tau_n+s_n;\tau_n,a_n)=v.
	\end{align*}
\end{remark}
\begin{remark}[periodic equations]
	For $\theta$-periodic difference equations \eqref{deq} and sets $\cA$ the fibres $\Omega_\cA(\tau)$ are $\theta$-periodic due to \eqref{noper} and \eqref{lem42a}. If $\cA$ is moreover positively invariant, then $\Omega_\cA$ are even constant and thus $\omega_\cA^-=\omega_\cA^+=\Omega_\cA(\tau)$ for all $\tau\in\Z$. 
\end{remark}

\begin{proof}[of \tref{thm46}] %\dontknow claimed in Section 3 of \cite{cui:kloeden:yang:19} - proved by Huy
	Since $\Omega_\cA(\tau)$ is nonempty, there exists a point $v\in\Omega_\cA(\tau)$ for all $\tau\in\I$. This implies that $v$ is also contained in $\omega_\cA^+$, i.e.,\ the forward $\omega$-limit set $\ocA$ is nonempty. From \lref{lem42}, we know that $\Omega_\cA(\tau)\subseteq K$ for each $\tau\in\I$. This yields $\ocA\subseteq K$. Moreover, since $K$ is compact and $\ocA$ is closed, $\ocA$ is also compact. The claimed limit relation is a consequence of \eqref{lem42b} and $\Omega_\cA(\tau)\subseteq\omega_{\cA}^+$. 
	
	The properties of $\omega_{\cA}^-$ are an immediate consequence of the fact that $\omega_{\cA}^-$ is an intersection of nested compact sets (cf.~\cite[p.~23, Lemma~22.2(5)]{sell:you:02}). 
	\qed
\end{proof}

Note that \eref{exninv} demonstrates that both the set $\omega_\cA^-$, as well as the forward limit sets $\omega_\cA^+$ constructed in \tref{thm46} are not invariant or even positively invariant. Yet, under additional assumptions weaker forms of invariance hold: 
\begin{theorem}[Asymptotic positive invariance]\label{thm49}
	Suppose that \eqref{deq} is $\cA$-asympto\-tically compact with a compact subset $K\subseteq U$ for a bounded, positively invariant $\cA\neq 0$. If for every sequence $\intoo{(s_n,\tau_n)}_{n\in\N}$ in $\N_0\tm\I$ with $\lim_{n\to\infty}\tau_n=\infty$, one has
	\begin{align*}
		\lim_{n\to\infty}\dist{\vphi(\tau_n+s_n;\tau_n,K),K}=0,
	\end{align*}
	then the forward $\omega$-limit set $\ocA$ is \emph{asymptotically positively invariant}, that is, for every strictly decreasing sequence $\eps_n\searrow 0$, there exists a strictly increasing sequence $T_n\nearrow\infty$ in $\I$ as $n\to\infty$ such that
	\begin{align}
		\vphi(\tau+s;\tau,\ocA)\subseteq B_{\eps_n}\intoo{\ocA}\fall T_n\leq\tau,\,s\in\N_0. 
		\label{thm44b}
	\end{align}
\end{theorem}
Recall the definition of the neighborhoods $B_{\eps_n}\intoo{\ocA}$ and thus \eqref{thm44b} reads as
\begin{align*}
	\dist{\vphi(\tau+s;\tau,\ocA),\ocA}<\eps_n \fall T_n\leq\tau,\,s\in\N_0.
\end{align*}

\begin{proof} %\dontknow Thm 3.2 from Cui, Kloeden, Yang - proved by Huy
	Suppose by contradiction that there exists a fixed $\eps_1>0$ so that there is a sequence $\intoo{(s_n^1,\tau_n)}_{n\in\N}$ with $0\leq s_n^1=s_n^1(\eps_1)\leq T_0(\tau_n,\eps_1)$ and $\tau_n\to\infty$ as $n\to\infty$ satisfying
	\begin{align}
		\dist{\vphi(\tau_n+s_n^1;\tau_n,\ocA),\ocA}\geq\eps_1 \fall n\in\N.
		\label{thm44d}
	\end{align}
	Since $\vphi$ is continuous and $\ocA$ is compact due to \tref{thm46}, $\vphi(\tau_n+s_n^1;\tau_n,\ocA)$ is also a compact set. This implies that there exists a 
	$$
		y_n^1=y_n^1(\eps_1):=\vphi(\tau_n+s_n^1;\tau_n,w_n^1) \in\vphi(\tau_n+s_n^1;\tau_n,\ocA)\subseteq\vphi(\tau_n+s_n^1;\tau_n,K)
	$$ 
	with $w_n^1=w_n^1(\eps_1)\in\ocA\subseteq K$ such that
	\begin{align*}
		\dist{y_n^1,\ocA}&=\dist{\vphi(\tau_n+s_n^1;\tau_n,w_n^1),\ocA}\\
		&=\dist{\vphi(\tau_n+s_n^1;\tau_n,\ocA),\ocA}\geq\eps_1 \fall n\in\N.
	\end{align*}
	On the other hand, with $y_n^1\in\vphi(\tau_n+s_n^1;\tau_n,K)$, by the assumption, we obtain
	\begin{align*}
		0\leq\dist{y_n^1,K}\leq\dist{\vphi(\tau_n+s_n^1;\tau_n,K),K}\xrightarrow[n\to\infty]{}0,
	\end{align*}
	implying that $y_n^1\in K$. Additionally, since the set $K$ is compact, there is a subsequence $(y_{n_j}(\eps_1))_{j\in\N}$ converging to $\bar{y}_1=\bar{y}_1(\eps_1)\in K$. Therefore, by definition, the inclusion $\bar{y}_1\in\Omega_\cA(\tau)\subseteq\ocA$ leads to $\bar{y}_1\in\ocA$, i.e.,\ $\dist{\bar{y}_1,\ocA}<\eps$ for all $\eps>0$, a contradiction to \eqref{thm44d}. Thus, for this $\eps_1>0$, there exists an integer $s_1=s_1(\eps_1)$ large enough such that
	\begin{align*}
		\dist{\vphi(\tau+s_1;\tau,\ocA),\ocA}<\eps_1.
	\end{align*}
	Repeating inductively with $\eps_{n+1}<\eps_n$ and $s_{n+1}(\eps_{n+1})>s_n(\eps_n)$ for all $n\in\N$, we then obtain that $\ocA$ is asymptotically positively invariant.
	\qed
\end{proof}

\begin{theorem}[Asymptotic negative invariance]\label{thm410}
	Suppose that \eqref{deq} is strongly $\cA$-asympto\-tically compact with a compact subset $K\subseteq U$ for a bounded, positively invariant $\cA\neq 0$. If for every $\eps>0$ and $T\in\N$, there exists a real $\delta=\delta(\eps,T)>0$ such that for all $\tau\in\I$, one has the implication
	\begin{align}
		\left.
		\begin{array}{c}
			u_0,v_0\in\cA(\tau)\cup K,\\
			d(u_0,v_0)<\delta
		\end{array}
		\right\}
		\Rightarrow
		\sup_{0\leq s\leq T} d\bigl(\vphi(\tau+s;\tau,u_0),\vphi(\tau+s;\tau,v_0)\bigr)<\eps,
		\label{thm47a}
	\end{align}
	then the forward $\omega$-limit set $\ocA$ is \emph{asymptotically negatively invariant}, that is, for all $u\in\ocA$, $\eps>0$ and $T\in\N$, there are integers $s^\ast=s^\ast(\eps)$ satisfying $\tau+s^\ast-T\in\I$ and $u_\eps^\ast\in\ocA$ such that
	$$
		d\bigl(\vphi(\tau+s^\ast;\tau+s^\ast-T,u_\eps^\ast),u\bigr)<\eps.
	$$
\end{theorem}
\begin{proof} % \dontknow Thm 3.4 from Cui, Kloeden, Yang - proved by Huy
	Consider reals $\eps>0$ and $T\in\N$ and take a point $u\in\ocA$. Thanks to \rref{rem45}, there is a sequence $\intoo{(s_n,\tau_n,a_n)}_{n\in\N}$ with $T<s_n=s_n(\eps)\to\infty$, $\tau_n\to\infty$ as $n\to\infty$, $\tau_n\in\I$, $\tau_n+s_n-T\in\I$ and $a_n=a_n(\eps)\in\cA(\tau_n)$, and an integer $N=N_1(\eps)$ with
	\begin{align*}
		d\bigl(\vphi(\tau_n+s_n;\tau_n,a_n),u\bigr)<\tfrac{\eps}{2} \fall n\geq N_1(\eps).
	\end{align*}
	Given a sequence $y_n:=\vphi(t^\ast_n-T;\tau_n,a_n)\in\vphi(t^\ast_n-T;\tau_n,\cA(\tau_n)) \subseteq \cA(t^\ast_n-T)$ with $t^\ast_n=t^\ast_n(\eps):=\tau_n+s_n$, $s_n-T\to\infty$, $\tau_n\to\infty$ as $n\to\infty$, $\tau_n\in\I$ and $a_n\in\cA(\tau_n)$, by the strong $\cA$-asymptotic compactness of \eqref{deq}, we obtain
	\begin{align*}
		0\leq\dist{y_n,K}\leq\dist{\vphi(t^\ast_n-T;\tau_n,\cA(\tau_n)), K}\xrightarrow[n\to\infty]{}0.
	\end{align*}
	Since $K$ is compact, $\dist{y_n,K}=\min_{k\in K} d(y_n,k)$. This implies that there exist a sequence $(k_n)_{n\in\N}$ in $K$ such that
	\begin{align*}
		d(y_n,k_n)=\dist{y_n,K}\xrightarrow[n\to\infty]{}0,
	\end{align*}
	and a subsequence $(k_{n_j})_{j\in\N}$ converging to $\bar{k}=\bar{k}(\eps)\in K$. Thus, 
	\begin{align*}
		0\leq
		d\intoo{y_{n_j},\bar{k}}
		\leq d\intoo{y_{n_j},k_{n_j}}+d\intoo{k_{n_j},\bar{k}}
		=\dist{y_n,K}+d\intoo{k_{n_j},\bar{k}}\xrightarrow[n\to\infty]{}0,
	\end{align*}
	implying $y_{n_j}:=\vphi(t^\ast_{n_j}-T;\tau_{n_j},a_{n_j})\xrightarrow[j\to\infty]{}\bar{k}$ with $t^\ast_{n_j}:=\tau_{n_j}+s_{n_j}$. Hence, by \rref{rem45}, one has $\bar{k}\in\ocA$. Moreover, with $y_{n_j}\in\cA(t^\ast_{n_j}-T)$ and $\bar{k}\in K$, by the assumption, we obtain for an integer $N_2(\eps,T)$ large enough,
	\begin{align*}
		d\bigl(\vphi(t^\ast_{n_j};t^\ast_{n_j}-T,y_{n_j}), \vphi(t^\ast_{n_j};t^\ast_{n_j}-T,\bar{k})\bigr)<\tfrac{\eps}{2} \fall n_j\geq N_2(\eps,T).
	\end{align*}
	Now the triangle inequality and the process property \eqref{semigroup} yield
	\begin{eqnarray*}
		&&
		d\bigl(\vphi(t^\ast_{n_j};t^\ast_{n_j}-T,\bar{k}),u\bigr)\\
		& \leq & d\bigl(\vphi(t^\ast_{n_j};t^\ast_{n_j}-T,\bar{k}), \vphi(t^\ast_{n_j};t^\ast_{n_j}-T,y_{n_j})\bigr)
		+ d\bigl(\vphi(t^\ast_{n_j};t^\ast_{n_j}-T,y_{n_j}),u\bigr) \\
		& = & d\bigl(\vphi(t^\ast_{n_j};t^\ast_{n_j}-T,\bar{k}), \vphi(t^\ast_{n_j};t^\ast_{n_j}-T,y_{n_j})\bigr)\\
		&&\qquad+d\bigl(\vphi(t^\ast_{n_j};t^\ast_{n_j}-T,\vphi(t^\ast_{n_j}-T;\tau_{n_j},a_{n_j})), u\bigr) \\
		& = & d\bigl(\vphi(t^\ast_{n_j};t^\ast_{n_j}-T,\bar{k}), \vphi(t^\ast_{n_j};t^\ast_{n_j}-T,y_{n_j})\bigr)
		+ d\bigl(\vphi(t^\ast_{n_j};\tau_{n_j},a_{n_j}),u\bigr) \\
		& < & \tfrac{\eps}{2}+\tfrac{\eps}{2}=\eps
		\fall n\geq N_1(\eps),n_j\geq N_2(\eps,T).
	\end{eqnarray*}
	Setting $u_\eps^\ast:=\bar{k}\in\ocA$, we then obtain that $\ocA$ is asymptotically negatively invariant. 
	\qed
\end{proof}
\subsection{Asymptotically autonomous difference equations}
In general it is difficult to obtain the forward limit set $\omega_\cA^+$ given as limit of the fibres $\Omega_\cA(\tau)$ explicitly. This situation simplifies, if \eqref{deq} behaves asymptotically as an autonomous difference equation
\begin{equation}
	\tag{$\Delta^\infty$}
	\boxed{u_{t+1}=\eF(u_t)}
	\label{deqa}
\end{equation}
with right-hand side $\eF:U\to U$ in a sense to be specified below. Here, it is common to denote the iterates of $\eF$ by $\eF^s:U\to U$, $s\in\N_0$. A maximal, invariant and nonempty compact set $A^\ast\subseteq U$ attracting all bounded subsets of $U$ is called \emph{global attractor} of \eqref{deqa} (cf.\ \cite[p.~17]{hale:88}). 

For a class of nonautonomous equations \eqref{deq} introduced next, the sets $\Omega_\cA(\tau)$, $\tau\in\I$, turn out to be constant and determined by the global attractor $A^\ast$ of \eqref{deqa}. 
\begin{theorem}[Asymptotically autonomous difference equations]\label{thmaa}
	Suppose that \eqref{deqa} has a bounded absorbing set $A\subseteq U$ and a global attractor $A^\ast\subseteq A$. If $\cA:=\I\tm A$ is a forward absorbing set of \eqref{deq} and the condition
	\begin{equation}
		\lim_{s\to\infty}\sup_{a\in A}d(\vphi(\tau+s;\tau,a),\eF^s(a))=0\fall\tau\in\I
		\label{thmaa1}
	\end{equation}
	holds, then $\Omega_{\cA}(\tau)=A^\ast$ for all $\tau\in\I$ and in particular $\omega_{\cA}^-=\omega_{\cA}^+=A^\ast$. 
\end{theorem}
\begin{remark}
	Asymptotically autonomous difference equations were also studied in \cite{cui:kloeden:19} in order to show that the fibres $\cA^\ast(\tau)$ of a pullback attractor $\cA^\ast$ to \eqref{deq} converge to the global attractor $A^\ast$ of the limit equation \eqref{deqa} as $\tau\to\infty$. In these results, however, asymptotic autonomy is based on e.g.\ the limit relation
	$$
		\lim_{\tau\to\infty}d(\vphi(\tau+s;\tau,a_\tau),\eF^s(a_0))=0\fall s\in\N_0
	$$
	(see \cite[Thm.~1]{cui:kloeden:19}) with sequences $(a_n)_{n\in\N}$ converging to some $a_0$. This condition is clearly different from \eqref{thmaa1}. 
\end{remark}
\begin{proof}
	Given any $\tau\in\I$, we have to show two inclusions: 

	$(\subseteq)$ Let $v\in\Omega_\cA(\tau)$. Due to \rref{rem43} there exist sequences $a_n\in A$ and $s_n\to\infty$ as $n\to\infty$ with
	$$
		\lim_{n\to\infty}\vphi(\tau+s_n;\tau,a_n)=v
	$$
	and it follows from \eqref{thmaa1} that
	\begin{eqnarray*}
		&&
		\dist{v,A^\ast}\\
		& \leq &
		d(v,\vphi(\tau+s_n;\tau,a_n))+d(\vphi(\tau+s_n;\tau,a_n),\eF^{s_n}(a_n))+\dist{\eF^{s_n}(a_n),A^\ast}\\
		& \leq &
		d(v,\vphi(\tau+s_n;\tau,a_n))+d(\vphi(\tau+s_n;\tau,a_n),\eF^{s_n}(a_n))
		+\dist{\eF^{s_n}(A),A^\ast}\xrightarrow[n\to\infty]{}0, 
	\end{eqnarray*}
	since the global attractor $A^\ast$ of \eqref{deqa} attracts the absorbing set $A$. This implies that $v\in A^\ast$, and since $v$ was arbitrary, the inclusion $\Omega_\cA(\tau)\subseteq A^\ast$ holds for $\tau\in\I$. 

	$(\supseteq)$ Conversely, since $A^\ast$ is compact, there exists an $a^\ast\in A^\ast$ with
	\begin{align*}
		\dist{A^\ast,\Omega_{\cA}(\tau)}
		&=
		\dist{a^\ast,\Omega_{\cA}(\tau)}\\
		&\leq
		\dist{a^\ast,\vphi(\tau+s;\tau,A^\ast)}+\dist{\vphi(\tau+s;\tau,A^\ast),\Omega_\cA(\tau)}
	\end{align*}
	for all $s\in\I$ and we separately estimate the two terms on the right-hand side of this inequality. First, due to the invariance of $A^\ast$ there exists $a_s^\ast\in A^\ast$ with $a^\ast=\eF^s(a_s^\ast)$ and therefore
	\begin{align*}
		\dist{a^\ast,\vphi(\tau+s;\tau,A^\ast)}
		&\leq
		d(a^\ast,\vphi(\tau+s;\tau,a_s^\ast))
		=
		d(\eF^s(a_s^\ast),\vphi(\tau+s;\tau,a_s^\ast))\\
		&\leq
		\sup_{a\in A^\ast}d(\vphi(\tau+s;\tau,a),\eF^s(a))
		\xrightarrow[s\to\infty]{\eqref{thmaa1}}0.
	\end{align*}
	Second, from $A^\ast\subseteq A=\cA(\tau)$ one has
	$$
		\dist{\vphi(\tau+s;\tau,A^\ast),\Omega_\cA(\tau)}
		\leq
		\dist{\vphi(\tau+s;\tau,\cA(\tau)),\Omega_\cA(\tau)}
		\xrightarrow[s\to\infty]{\eqref{lem42b}}0,
	$$
	which guarantees the remaining inclusion $A^\ast\subseteq\Omega_\cA(\tau)$. 
	
	Hence, all $\Omega_\cA(\tau)$ are constant, thus $\omega_\cA^-=A^\ast$ and $\omega_\cA^+=\overline{A^\ast}=A^\ast$. 
	\qed
\end{proof}

The following simple example illustrates the condition \eqref{thmaa1}: 
\begin{example}[Beverton-Holt equation]
	If $\alpha>1$, then it is well known that all solutions to the autonomous Beverton-Holt equation $v_{t+1}=\tfrac{\alpha v_t}{1+v_t}$ starting with a positive initial value converge to $\alpha-1$ (see \cite[pp.~13ff]{lutscher:19}). We establish that an asymptotically autonomous, but nonautonomous Beverton-Holt equation 
	\begin{align}
		v_{t+1}&=\frac{\tilde a_tv_t}{1+v_t},&
		\tilde a_t&:=\frac{f_{t+1}}{f_t}\alpha
		\label{exbhasy1}
	\end{align}
	shares this behaviour, whenever the sequence $(f_t)_{t\in\I}$ in $(0,\infty)$ satisfies
	\begin{equation}
		\lim_{t\to\infty}\sum_{s=0}^{t-1}\frac{f_{\tau+s}}{f_{\tau+t}}\alpha^{s-t}=\frac{1}{\alpha-1}
		\fall\tau\in\I
		\label{exbhasy3}
	\end{equation}
	and grows at most polynomially. For instance, the relation \eqref{exbhasy3} holds for the sequences $f_t=(t+c)^n$ with $c\in\R$, $n\in\set{1,2,3,4}$. Indeed, the explicit representation
	$$
		\vphi(\tau+t;\tau,a)
		=
		\frac{(\alpha-1)a}{\tfrac{\alpha-1}{\alpha^t}\tfrac{f_\tau}{f_{\tau+t}}+(\alpha-1)a\sum_{s=0}^{t-1}\tfrac{f_{\tau+s}}{f_{\tau+t}}\alpha^{s-t}}
	$$
	of the general solution to \eqref{exbhasy1} yields that $\lim_{t\to\infty}\vphi(t+\tau;\tau,a)=\alpha-1$ holds uniformly in $a\in[\bar a,\infty)$ for any $\bar a>0$. Consequently, one has
	\begin{align*}
		\sup_{a\geq\bar a}\abs{\vphi(\tau+t;\tau,a)-F^t(a)}
		\leq&
		\sup_{a\geq\bar a}\abs{\vphi(\tau+t;\tau,a)-(\alpha-1)}+
		\sup_{a\geq\bar a}\abs{\alpha-1-F^t(a)}\\
		&\xrightarrow[t\to\infty]{}0
	\end{align*}
	and therefore \eqref{thmaa1} is valid with arbitrary subsets $A\subseteq[\bar a,\infty)$. 
\end{example}

We continue with two sufficient criteria for the condition \eqref{thmaa1} to hold. Thereto we assume in the remaining subsection that $(X,\norm{\cdot})$ is a Banach space. 
\begin{theorem}[Asymptotically autonomous linear difference equations]
	Suppose that $\eL_t,\eL\in L(X)$, $b_t,b\in X$, $t\in\I$, satisfy
	\begin{align}
		\lim_{t\to\infty}\eL_t&=\eL,&
		\lim_{t\to\infty}b_t&=b.
		\label{exlin1}
	\end{align}
	If $\rho(\eL)<1$, then \eqref{deqa} with right-hand side $\eF(u)=\eL u+b$ and \eqref{deq} with right-hand side $\eF_t(u)=\eL_tu+b_t$ fulfill the limit relation \eqref{thmaa1} on every bounded subset $A\subseteq X$.
\end{theorem}
\begin{proof}
	We note that \eqref{exlin1} implies that \eqref{deq} is uniformly exponentially stable and that the sequences $(\eL_t)_{t\in\I}$, $(b_t)_{t\in\I}$ are bounded. Thus, \cite[Cor.~5 with $H=\ell^\infty(\I,X)$]{aulbach:minh:96} implies that $\vphi(\cdot;\tau,0)$ is bounded and the representation $\vphi(t;\tau,u_\tau)=\Phi(t,\tau)u_\tau+\vphi(t;\tau,0)$ for all $t\geq\tau$, $u_\tau\in X$ shows that $\vphi(\cdot;\tau,u_\tau)$ is bounded uniformly in $u_\tau$ from bounded subsets of $X$. Now it is easy to see that the difference $\vphi(t;\tau,a)-\eF^{t-\tau}(a)$ solves the initial value problem
	\begin{align*}
		w_{t+1}&=\eL w_t+\tilde b_t,&
		u_\tau&=0,
	\end{align*}
	whose inhomogeneity $\tilde b_t:=(\eL_t-\eL)\vphi(t;\tau,a)+b_t-b$ satisfies $\lim_{t\to\infty}\tilde b_t=0$ uniformly in $a$ from bounded subsets of $X$. Now using \cite[Cor.~5 with $H=c_0(\I,X)$]{aulbach:minh:96} guarantees that the sequence $\vphi(t;\tau,a)-\eF^{t-\tau}(a)$ converges to $0$ as $t\to\infty$ uniformly in $a$ from bounded subsets of $X$, that is, in particular \eqref{thmaa1} holds. 
	\qed
\end{proof}

\begin{theorem}[Asymptotically autonomous semilinear difference equations]\label{thmaasl}
	Let $\eF_t:U\to U$ be of semilinear form \eqref{rhssemi} such that
	\begin{equation}
		\norm{\Phi(t,s)}\leq K\alpha^{t-s}\fall s\leq t
		\label{thmaasl1}
	\end{equation}
	and
	\begin{equation}
		\norm{\eN_t(u)-\eN_t(\bar u)}\leq L\norm{u-\bar u}\fall t\in\I,\,u,\bar u\in U
		\label{thmaasl2}
	\end{equation}
	with $K\geq 1$, $\alpha\in(0,1)$ and $L\in(0,\tfrac{1-\alpha}{K})$. If $\eL\in L(X)$ and $\eN:U\to X$ satisfy
	\begin{itemize}
		\item[(i)] there exists a $K_1\geq 0$ such that $\norm{\eL_t-\eL}\leq K_1\alpha^t$ for all $t\in\I$, 

		\item[(ii)] for every $r>0$ there exists a $K_2(r)\geq 1$ such that
		$$
			\sup_{u\in U\cap B_r(0)}\norm{\eN_t(u)-\eN(u)}\leq K_2(r)\alpha^t\fall t\in\I, 
		$$
	\end{itemize}
	then \eqref{deqa} with right-hand side $\eF(u)=\eL u+\eN(u)$ and \eqref{deq} fulfill the limit relation \eqref{thmaa1} even exponentially on every bounded subset $A\subseteq X$.
\end{theorem}
The assumption \eqref{thmaasl1} holds in case $\rho(\eL)<1$ with $\alpha\in(\rho(\eL),1)$ (see \cite[Thm.~5]{przyluski:88}). 
\begin{proof}
	We proceed in two steps:

	(I) Claim: \emph{All solutions to the autonomous equation \eqref{deqa} are bounded, i.e.,\ 
	\begin{equation}
		\norm{\eF^t(a)}\leq K\norm{a}+\frac{K}{1-\alpha-KL}\norm{\eN(0)}
		\fall t\geq 0,\,a\in X.
	\end{equation}}
	This is a consequence of \cite[p.~155, Thm.~3.5.8(a)]{poetzsche:10b}.

	(II) Let $a\in U$ and we abbreviate $u_t:=\vphi(t;\tau,a)$, $v_t:=\eF^{t-\tau}(a)$. Due to step (I) the sequence $(v_t)_{t\geq\tau}$ is bounded and we choose $r>0$ so large that $\norm{v_t}<r$ for all $t\in\I$. It is easy to see that the difference $w_t:=u_t-v_t$ satisfies the equation
	\begin{equation}
		w_{t+1}=\eL_tw_t+(\eL_t-\eL)v_t+\eN_t(v_t)-\eN(v_t)+\eN_t(u_t)-\eN_t(v_t)
	\end{equation}
	and fulfills the initial condition $u_\tau-v_\tau=a-a=0$. Using the variation of constants formula \cite[p.~100, Thm.~3.1.16]{poetzsche:10b} results
	$$
		w_t
		=
		\sum_{s=\tau}^{t-1}\Phi(t,s+1)\intcc{(\eL_s-\eL)v_s+\eN_s(v_s)-\eN(v_s)+\eN_s(u_s)-\eN_s(v_s)}
		\fall t\geq\tau, 
	$$
	consequently
	\begin{eqnarray*}
		&&
		\norm{w_t}\alpha^{-t}\\
		& \stackrel{\eqref{thmaasl1}}{\leq} &
		\frac{K}{\alpha}\sum_{s=\tau}^{t-1}\alpha^{-s}\norm{(\eL_s-\eL)v_s+\eN_s(v_s)-\eN(v_s)}
		+
		\frac{K}{\alpha}\sum_{s=\tau}^{t-1}\alpha^{-s}\norm{\eN_s(u_s)-\eN_s(v_s)}\\
		& \stackrel{\eqref{thmaasl2}}{\leq} &
		\frac{K}{\alpha}\sum_{s=\tau}^{t-1}\alpha^{-s}\norm{(\eL_s-\eL)v_s+\eN_s(v_s)-\eN(v_s)}
		+
		\frac{KL}{\alpha}\sum_{s=\tau}^{t-1}\alpha^{-s}\norm{w_s}
	\end{eqnarray*}
	and the Gr\"onwall inequality \cite[p.~348, Prop.~A.2.1(a)]{poetzsche:10b} yields
	$$
		\norm{w_t}
		\leq
		\frac{K}{\alpha}\sum_{s=\tau}^{t-1}(\alpha+KL)^{t-s-1}\norm{(\eL_s-\eL)v_s+\eN_s(v_s)-\eN(v_s)}
		\fall\tau\leq t
	$$
	If we replace $t$ by $\tau+t$, then it results
	\begin{eqnarray*}
		&&
		\norm{\vphi(\tau+t;\tau,a)-\eF^t(a)}\\
		& \leq &
		\frac{K}{\alpha}\sum_{s=\tau}^{t+\tau-1}(\alpha+KL)^{t+\tau-s-1}\norm{(\eL_s-\eL)v_s+\eN_s(v_s)-\eN(v_s)}\\
		& \stackrel{(i)}{\leq} &
		\frac{K}{\alpha}\sum_{s=\tau}^{t+\tau-1}(\alpha+KL)^{t+\tau-s-1}\intoo{K_1r\alpha^s+\norm{\eN_s(v_s)-\eN(v_s)}}\\
		& \stackrel{(ii)}{\leq} &
		\frac{K\intoo{K_1r+K_2(r)}}{\alpha}(\alpha+KL)^{t+\tau-1}
		\sum_{s=\tau}^{t+\tau-1}\intoo{\tfrac{\alpha}{\alpha+KL}}^s
		\xrightarrow[t\to\infty]{}0
	\end{eqnarray*}
	and therefore even exponential convergence holds in \eqref{thmaa1}. 
	\qed
\end{proof}
\subsection{Forward attractors}
\label{sec43}
In the previous \sref{sec32} we constructed pullback attractors of nonautonomous difference equations \eqref{deq} by means of \pref{prop31} applied to a pullback absorbing, positively invariant nonautonomous set. Now it is our goal is to obtain a corresponding concept in forward time. 

Mimicking the approach for pullback attractors, we define a \emph{forward attractor} $\cA^+\subset\cU$ of \eqref{deq} as a nonempty, compact and invariant nonautonomous set forward attracting every bounded subset $\cB\subseteq\cU$, i.e.,\ 
\begin{align}
	\lim_{s\to\infty}\dist{\vphi(\tau+s;\tau,\cB(\tau)),\cA^+(\tau+s)}=0 \fall\tau\in\I. 
	\label{nofa}
\end{align}
As demonstrated in e.g.\ \cite[Sect.~4]{kloeden:yang:16}, forward attractors need not to be unique. They are Lyapunov asymptotically stable, that is, Lyapunov stable and attractive in the sense of \eqref{nofa} (see \cite[Prop.~3.1]{kloeden:lorenz:16}). While it is often claimed in the literature that there is no counterpart to the characterisation \eqref{charpb} of pullback attractors $\cA^\ast$ for forward attractors $\cA^+$ of nonautonomous equations, a suitable construction will be given now. 

A nonautonomous difference equation \eqref{deq} is denoted as \emph{forward dissipative}, if there exists a bounded set $\cA\subseteq\cU$ such that for every $\tau\in\I$ and bounded $\cB\subseteq\cU$ there is an \emph{absorption time} $S=S(\tau,\cB)\in\N$ such that 
\begin{equation}
	\vphi(\tau+s;\tau,\cB(\tau))\subseteq\cA(\tau+s)\fall s\geq S;
	\label{fordiss}
\end{equation}
one says that $\cA$ is as a \emph{forward absorbing} set. 
\begin{proposition}[Forward absorbing set]\label{propfwab}
	On a Banach space $X$, let $\eF_t:U\to U$ be of the semilinear form \eqref{rhssemi} satisfying \eqref{lemabs1}, \eqref{lemabs2} and let $\rho>0$. If the limit relations
	\begin{align*}
		\lim_{s\to\infty}\prod_{r=\tau}^{\tau+s-1}(\alpha_r+Ka_r)&=0,&
		R_\tau&:=K\lim_{t\to\infty}\sum_{s=\tau}^{t-1}b_s\prod_{r=s+1}^{t-1}(\alpha_r+Ka_r)<\infty
	\end{align*}
	hold for all $\tau\in\I$ and $\sup_{\tau\in\I}R_\tau<\infty$, then the difference equation \eqref{deq} is forward dissipative with absorbing set $\cA:=\set{(t,u)\in\cU:\,\norm{u}\leq\rho+\sup_{\tau\in\I}R_\tau}$. 
\end{proposition}
For constant positive sequences $\alpha_t\equiv\alpha$, $a_t\equiv a$, $b_t\equiv b$ in \eqref{lemabs2} satisfying $\alpha+Ka<1$, both the pullback absorbing set from \pref{proppbab} and the forward absorbing set from \pref{propfwab} have constant fibres and simplify to $\cA=\I\tm B_{\rho+\tfrac{Kb}{1-\alpha-aK}}(0)$. 
\begin{proof}
	The claim follows readily from relation \eqref{lemabs3} in \lref{lemabs}.
	\qed
\end{proof}

Using \cite[Prop.~3.2 with compact replaced by open and bounded]{kloeden:lorenz:16} one shows
\begin{proposition}\label{prop417}
	Every bounded forward attractor has a nonempty, positively invariant, closed and bounded forward absorbing set. 
\end{proposition}

First, this \pref{prop417} allows us to choose a closed and bounded, positively invariant set $\cA\subseteq\cU$. We then deduce a nonempty, invariant and compact nonautonomous set $\cA^\star\subseteq\cA$ from \pref{prop31}. 

Second, the construction of forward attractors requires $\I=\Z$. Different from the pullback situation (with $\cA$ being pullback absorbing), having an forward absorbing set $\cA$ does not ensure the forward convergence within $\cA$, i.e.,\ 
$$
	\lim_{s\to\infty}\dist{\vphi(\tau+s;\tau,\cA(\tau)),\cA^\star(\tau)}=0\fall\tau\in\Z
$$
and in particular not forward convergence of a general bounded nonautonomous set $\cB\subseteq\cU$ to $\cA^\ast$. This is because \eqref{deq} may have forward limit points starting in $\cA$ which are not forward limit points from within $\cA^\star$. Corresponding examples illustrating this are given in \cite{kloeden:poetzsche:rasmussen:11}. 

Now on the one hand, the set of forward $\omega$-limit points for the dynamics starting in $\cA^\star$ is given by
$$
	\omega_{\cA}^\star
	:=
	\bigcap_{\tau\in\Z} \overline{\bigcup_{0\leq s}\vphi(\tau+s;\tau,\cA^\star(\tau))}
	=
	\bigcap_{\tau\in\Z}\overline{\bigcup_{0\leq s}\cA^\star(\tau+s)}
	\subseteq U
$$
and is nonempty and compact as intersection of nested compact sets. It consists of all points $u\in U$ such that there is a sequence $\intoo{(s_n,a_n)}_{n\in\N}$ with $\lim_{n\to\infty}s_n=\infty$ and $a_n\in\cA(\tau+s_n)$ with $\tau\in\Z$ satisfying
\begin{align*}
	\lim_{n\to\infty}\vphi(\tau+s_n;\tau,a_n)=u.
\end{align*}

On the other hand, the set of forward limit points $\omega_{\cA}^+$ from within $\cA$ was constructed in \tref{thm46}. With $\cA$ being positively invariant, the chain of inclusions $\omega_{\cA}^-\subseteq\omega_{\cA}^+\subseteq K$ holds, while $\omega_{\cA}^\star$ is not necessarily contained in $\omega_\cA^-$ (see \eref{exbhex} for an illustration), as well as
\begin{align*}
	\lim_{t\to\infty}\dist{\cA^\star(t),\omega_{\cA}^-}=0.
\end{align*}
\begin{theorem}\label{thm418}
	Suppose that \eqref{deq} has a positively invariant, closed and bounded set $\cA\neq\emptyset$. If the assumptions in \pref{prop31} and Thms.~\ref{thm49}--\ref{thm410} hold, then the following statements are equivalent:
	\begin{enumerate}
		\item $\cA^\star$ is \emph{forward attracting} $\cA$, that is, 
		\begin{align}
			\lim_{s\to\infty}\dist{\vphi(\tau+s;\tau,\cA(\tau)),\cA^\star(\tau+s)}=0\fall\tau\in\Z,
			\label{thm49a}
		\end{align}
		
		\item $\omega_{\cA}^+=\omega_{\cA}^\star$.
	\end{enumerate}
\end{theorem}
\begin{proof} %	\dontknow Thm 5.2 from Cui, Kloeden, Yang - proved by Huy
	($\Rightarrow$) Suppose that $\cA^\star$ is forward attracting from within $\cA$ and that $\omega_{\cA}^+\neq\omega_{\cA}^\star$. Since $\omega_{\cA}^\star\subseteq\omega_{\cA}^+$, there exists a point $\tilde{v}\in\omega_{\cA}^+\setminus\omega_{\cA}^\star$, i.e.,\ there are $\tilde{\tau}\in\Z$ and $\tilde{\eps}=\tilde{\eps}(\tilde{\tau})>0$ such that $\tilde{v}\in\Omega_\cA(\tilde{\tau})$ and 
	\begin{align}
		\dist{\tilde{v},\omega_{\cA}^\star}> 2\tilde{\eps}.
		\label{thm49c}
	\end{align}
	Since $\tilde{v}\in\Omega_\cA(\tilde{\tau})$, there exists a sequence $\intoo{(\tilde{s}_n,\tilde{b}_n)}_{n\in\N}$ with $\lim_{n\to\infty}\tilde{s}_n=\infty$ and points $\tilde{b}_n\in\cA(\tilde{\tau})$ satisfying $\dist{\tilde{v},\vphi(\tilde{\tau}+\tilde{s}_n;\tilde{\tau},\tilde{b}_n)}<\tilde{\eps}$. Moreover, by the forward attraction of $\cA^\star$, there exists an $s'>0$ such that
	\begin{align*}
		\dist{\vphi(\tilde{\tau}+s';\tilde{\tau},\cA(\tilde{\tau})),\cA^\star(\tilde{\tau}+s')}<\tilde{\eps}.
	\end{align*}
	Combining all of them, we obtain
	\begin{eqnarray*}
		\dist{\tilde{v},\cA^\star(\tilde{\tau}+\tilde{s}_n)}
		& \leq &
		\dist{\tilde{v},\vphi(\tilde{\tau}+\tilde{s}_n;\tilde{\tau},\tilde{b}_n)} + \dist{\vphi(\tilde{\tau}+\tilde{s}_n;\tilde{\tau},\tilde{b}_n),\cA^\star(\tau+\tilde{s}_n)}\\
		& \leq &
		\dist{\tilde{v},\vphi(\tilde{\tau}+\tilde{s}_n;\tilde{\tau},\tilde{b}_n)} + \dist{\vphi(\tilde{\tau}+\tilde{s}_n;\tilde{\tau},\cA(\tilde{\tau})),\cA^\star(\tau+\tilde{s}_n)}\\
		& < &
		\tilde{\eps}+\tilde{\eps}=2\tilde{\eps}.
	\end{eqnarray*}
	Since $\bigcap_{n\in\N}\overline{\bigcup_{m\geq n}\cA^\star(\tilde{\tau}+\tilde{s}_n)}\subseteq\omega_{\cA}^\star$ by definition, it then follows
	\begin{align*}
		\dist{\tilde{v},\omega_{\cA}^\star}
		\leq
		\operatorname{dist}\left(\tilde{v},\bigcap_{n\in\N}\overline{\bigcup_{m\geq n} \cA^\star(\tilde{\tau}+\tilde{s}_n)}\right)
		\leq
		\dist{\tilde{v},\cA^\star(\tilde{\tau}+\tilde{s}_n)} < 2\tilde{\eps},
	\end{align*}
	a contradiction to \eqref{thm49c}. Hence, $\omega_{\cA}^+=\omega_{\cA}^\star$ holds.
	
	($\Leftarrow$) Suppose that $\omega_{\cA}^+=\omega_{\cA}^\star$, i.e.,\ $\dist{\omega_{\cA}^+,\omega_{\cA}^\star}<\eps$ for all $\eps>0$, and that $\cA^\star$ is not forward attracting from within $\cA$, i.e.,\ there exist a real $\tilde{\eps}>0$ and a sequence $(\tilde{s}_n)_{n\in\N}$ in $\N_0$ with $\lim_{n\to\infty}\tilde{s}_n=\infty$ and
	\begin{align*}
		\dist{\vphi(\tau+\tilde{s}_n;\tau,\cA^\star(\tau)),\cA(\tau+\tilde{s}_n)}\geq 2\tilde{\eps} \fall n\in\N.
	\end{align*}
	Although there is no condition ensuring the set $\vphi(\tau+s_n;\tau,\cA(\tau))$ is compact, which means the supremum in the Hausdorff semidistance may not be attained, there still exists a point $\tilde{y}_n:=\vphi(\tau+s_n;\tau,\tilde{b}_n) \in\vphi(\tau+s_n;\tau,\cA(\tau))$ for all $n\in\N$ and $\tilde{b}_n\in\cA(\tau)$ such that
	\begin{align*}
		\dist{\vphi(\tau+\tilde{s}_n;\tau,\cA(\tau)),\cA^\star(\tau+\tilde{s}_n)}-\tilde{\eps} 
		&\leq
		\dist{\tilde{y}_n,\cA^\star(\tau+\tilde{s}_n)}\\
		&\leq
		\dist{\vphi(\tau+\tilde{s}_n;\tau,\cA(\tau)},\cA^\star(\tau+\tilde{s}_n)).
	\end{align*}
	The above inequalities in fact give $\dist{\tilde{y}_n,\cA^\star(\tau+\tilde{s}_n)}\geq\tilde{\eps}$ for all $n\in\N$. Moreover, take a point $\tilde{a}_n\in\cA^\star(\tau+\tilde{s}_n)$, then
	\begin{align*}
		d(\tilde{y}_n,\tilde{a}_n)\geq\dist{\tilde{y}_n,\cA(\tau+\tilde{s}_n)}\geq\tilde{\eps} \fall n\in\N.
	\end{align*}
	On the other hand, by assumptions and definitions, \eqref{deq} is $\cA$-asymptotically compact and both $\tilde{y}_n$ and $\tilde{a}_n$ are in $\vphi(\tau+s_n;\tau,\cA(\tau))$ for all $\tau\in\Z$, so
	\begin{align*}
		\dist{\tilde{y}_n,K}&\leq \dist{\vphi(\tau+\tilde{s}_n;\tau,\cA(\tau)),K} \xrightarrow[n\to\infty]{}0,\\
		\dist{\tilde{a}_n,K}&\leq \dist{\vphi(\tau+\tilde{s}_n;\tau,\cA(\tau)),K} \xrightarrow[n\to\infty]{}0,
	\end{align*}
	implying that both $\tilde{y}_n$ and $\tilde{a}_n$ are in $K$ as well. Additionally, since $K$ is compact, there are convergent subsequences $(\tilde{y}_{n_j})_{j\in\N}$ with limit $\tilde{y}\in K$ and $(\tilde{a}_{n_j})_{j\in\N}$ with limit $\tilde{a}\in K$. This implies $\tilde{y}\in\Omega_\cA(\tau)\subseteq\omega_{\cA}^+$ and $\tilde{a}\in\omega_{\cA}^\star$ by definitions. Combining this with $d(\tilde{y}_n,\tilde{a}_n)\geq\tilde{\eps}$ for all $n\in\N$, we arrive at the contradiction
	\begin{align*}
		\dist{\omega_{\cA}^+,\omega_{\cA}^\star}\geq\dist{\tilde{y}_n,\omega_{\cA}^\star}\geq d(\tilde{y},\tilde{a})\geq\tilde{\eps} \fall n\in\N,
	\end{align*}
	to the assumption. Thus, $\cA^\star$ is forward attracting from within $\cA$.	
	\qed
\end{proof}

\begin{corollary}\label{corfw}
	Suppose in addition that $\cA\subseteq\cU$ is forward absorbing. If $\omega_{\cA}^+=\omega_{\cA}^\star$ holds, then $\cA^\star$ is a forward attractor of \eqref{deq}. 
\end{corollary}
\begin{proof}
	Due to \pref{prop31} the set $\cA^\star$ is already nonempty, compact, invariant and thus it suffices to show that $\cA^\star$ is forward attracting. Thereto, suppose that $\cB\subseteq\cU$ is bounded and choose $\tau\in\Z$ arbitrarily. With the forward absorption time $S\in\N$ we obtain from \tref{thm418} that
	\begin{eqnarray*}
		0
		&\leq&
		\dist{\vphi(\tau+s;\tau,\cB(\tau)),\cA^\star(\tau+s)}\\
		& \stackrel{\eqref{semigroup}}{=} &
		\dist{\vphi\bigl(\tau+s;\tau+S,\vphi(\tau+S,\tau,\cB(\tau))\bigr),\cA^\star(\tau+s)}\\
		& \stackrel{\eqref{fordiss}}{\leq} &
		\dist{\vphi(\tau+s;\tau+S,\cA(\tau+S)),\cA^\star(\tau+s)}
		\xrightarrow[s\to\infty]{\eqref{thm49a}}0
	\end{eqnarray*}
	and this yields the assertion. 
	\qed
\end{proof}

We close this section with a simple, yet illustrative example: 
\begin{example}[Beverton-Holt equation]\label{exbhex}
	Given reals $0<\alpha_-,\alpha_+$ we consider the asymptotically autonomous Beverton-Holt equation
	\begin{align}
		v_{t+1}&=\frac{\tilde a_tv_t}{1+v_t},&
		\tilde a_t&:=
		\begin{cases}
			\alpha_-,&t<0,\\
			\alpha_+,&0\leq t
		\end{cases}
		\label{nobhx}
	\end{align}
	in $U=\R_+$ having the general solution
	$$
		\vphi(t;\tau,v_\tau)
		=
		\frac{v_\tau\prod_{r=\tau}^{t-1}\tilde a_r}{1+v_\tau\sum_{s=\tau}^{t-1}\prod_{r=\tau}^{s-1}\tilde a_r}
		\fall\tau\leq t,\,0\leq v_\tau. 
	$$
	It possesses the absorbing set $\cA=\Z\tm[0,\max\set{\alpha_-,\alpha_+}+1]$ and the forward $\omega$-limit set $\omega_\cA^+=[0,\max\set{0,\alpha_+-1}]$. Depending on the constellation of the parameters $\alpha_-,\alpha_+$ one obtains the following capturing the forward dynamics:\\
	\begin{tabular}{c|c|c|c|c}
	$\alpha_-,\alpha_+$ & $\cA^\ast=\cA^\star$ & $\omega_{\cA}^\star$ & $\omega_{\cA}^-$ & $\omega_{\cA}^+$\\
	\hline
	$\alpha_-,\alpha_+\leq 1$ &\includegraphics[width=30mm]{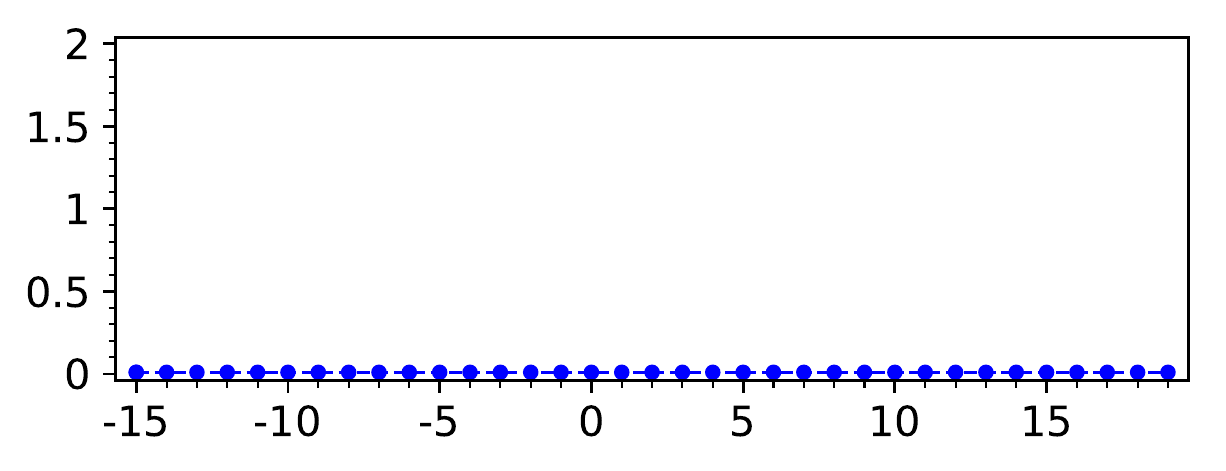} & $\set{0}$ & $\set{0}$ & $\set{0}$\\
	$\alpha_-\leq 1<\alpha_+$ &\includegraphics[width=30mm]{pba1} & $\set{0}$ & $\set{0}$ & $[0,\alpha_+-1]$\\
	$\alpha_+\leq 1<\alpha_-$ &\includegraphics[width=30mm]{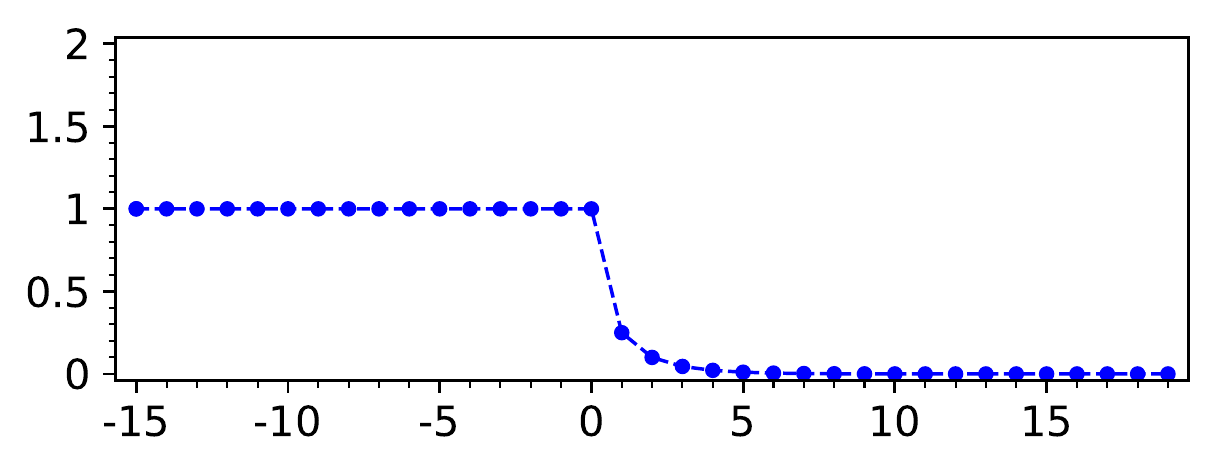} & $\set{0}$ & $\set{0}$ & $\set{0}$\\
	$1<\alpha_-<\alpha_+$ &\includegraphics[width=30mm]{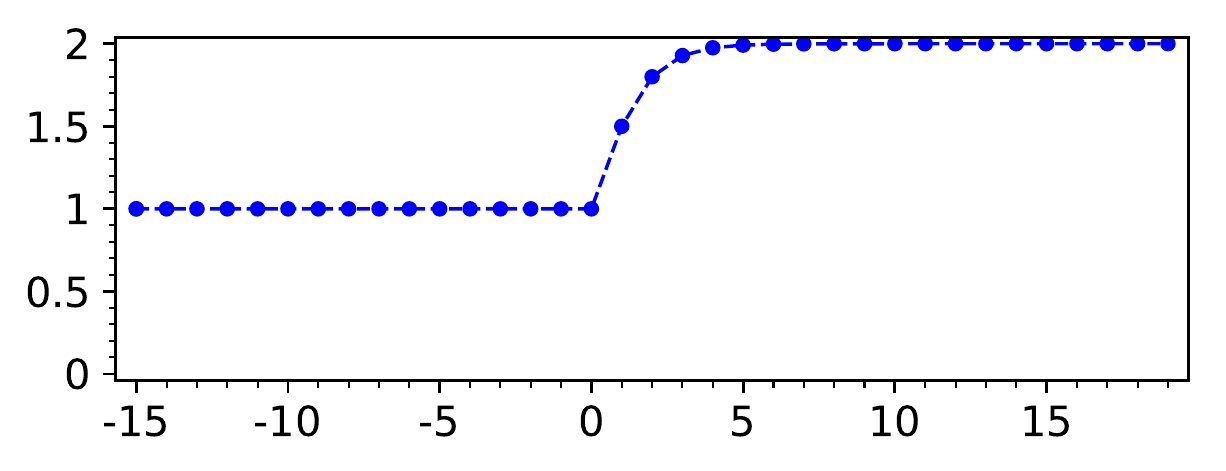} & $[0,\alpha_+-1]$ & $[0,\alpha_--1]$ &$[0,\alpha_+-1]$\\
	$1<\alpha_+\leq\alpha_-$ &\includegraphics[width=30mm]{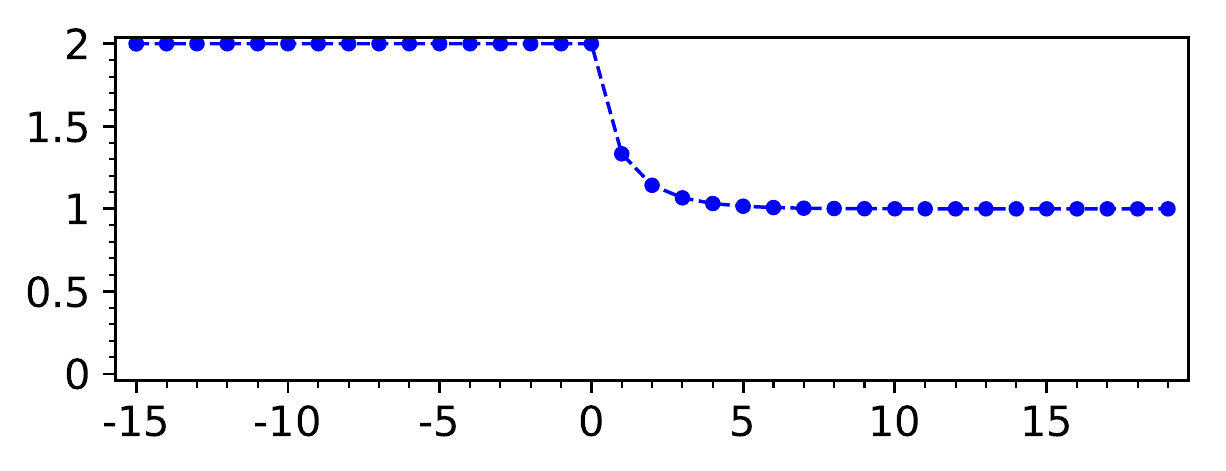} & $[0,\alpha_+-1]$ & $[0,\alpha_+-1]$ & $[0,\alpha_+-1]$\\
	\end{tabular}\\
	For $\alpha_+\leq 1$ all fibres $\Omega_\cA(\tau)=\set{0}$ are constant. For $\alpha_+>1$ two cases arise:
	\begin{itemize}
		\item $1\leq\alpha_-<\alpha_+$: Solutions starting in $\alpha_++1$ at time $\tau<0$ first decay until time $t=0$ and then increase again, which yields
		$$
			\Omega_{\cA}(\tau)
			=
			\begin{cases}
				[0,\vphi(0;\tau,\alpha_++1)],&\tau<0,\\
				[0,\alpha_+-1],&0\leq\tau. 
			\end{cases}
		$$

		\item $1\leq\alpha_+\leq\alpha_-$: Solutions starting in $\alpha_++1$ decay to $\alpha_+-1$ and thus the fibres are constant $\Omega_\cA(\tau)\equiv [0,\alpha_++1]$ on $\Z$. 
	\end{itemize}
	Except for $\alpha_-\leq 1<\alpha_+$, where the pullback and forward dynamics of \eqref{nobhx} differ, \cref{corfw} applies and yields that the pullback attractor $\cA^\ast=\cA^\star$ is the forward attractor $\cA^+$. 
\end{example}

As a conclusion, in case $\cA$ is a positively invariant, forward absorbing nonautonomous set this section provided two concepts to capture the forward dynamics of \eqref{deq}, namely the limit set $\omega_{\cA}^+$ from \tref{thm46} and the forward attractor $\cA^+=\cA^\star$ constructed in \cref{corfw}. On the one hand, the limit set $\omega_{\cA}^+\subseteq U$ is asymptotically positively invariant, forward attracts and is contained in all other sets with these properties. It depends only on information in forward time. On the other hand, the forward attractor $\cA^+\subseteq\cU$ shares these properties, but is actually invariant. Its construction is based on information on the entire axis $\Z$ and more restrictively, relies on the condition $\omega_{\cA}^+=\omega_{\cA}^\star$ from \tref{thm418}(b). The latter might be hard to verify in concrete examples, unless rather strict assumptions like asymptotic autonomy hold \cite{cui:kloeden:19}. 
\section{Integrodifference equations}
\label{sec5}
The above abstract results will now be applied to nonautonomous IDEs. For this purpose let $(\Omega,{\mathfrak A},\mu)$ be a measure space satisfying $\mu(\Omega)<\infty$. Suppose additionally that $\Omega$ is equipped with a metric such that it becomes a compact metric space. 

We consider the Banach space $X=\Cd$ of continuous $\R^d$-valued functions over $\Omega$ equipped with the norm
$
	\norm{u}_0:=\max_{x\in\Omega}\abs{u(x)}. 
$
If $Z\subseteq\R^d$ is a nonempty, closed set, then $U:=\bigl\{u:\Omega\to Z\mid u\in\Cd\bigr\}$ is a complete metric space. Furthermore, we have $\cU=\I\tm U$, where $\I$ is an unbounded discrete interval. 

Given functions $g_t:\Omega\tm Z\to\R^d$ and $k_t:\Omega^2\tm Z\to\R^d$, the \emph{Nemytskii operator} $\eG_t:U\to\Cd$ is defined by 
$$
	\eG_t(u)(x):=g_t(x,u(x))\fall(t,x)\in\I'\tm\Omega
$$
and the \emph{Urysohn integral operators} $\eK_t:U\to\Cd$ by 
$$
	\eK_t(u)(x):=\int_\Omega k_t(x,y,u(y))\d\mu(y)\fall(t,x)\in\I'\tm\Omega.
$$

With these operators, a nonautonomous difference equation \eqref{deq} of the additive form \eqref{rhsadd} is called an \emph{integrodifference equation} and explicitly reads as
\begin{equation}
	\tag{$I_g$}
	\boxed{
	u_{t+1}(x)
	=
	g_t(x,u_t(x))+\int_\Omega k_t(x,y,u_t(y))\d\mu(y)}
	\fall(t,x)\in\I'\tm\Omega.
	\label{ide}
\end{equation}
Such problems are well-motivated from applications: 
\begin{itemize}
	\item For an integrodifferential equation $D_1u(t,x)=\int_\Omega f(t,x,y,u(t,y))\d y$ with, e.g., a continuous kernel function $f:\R\tm\Omega^2\tm\R^d\to\R^d$, the forward Euler discretisation with step-size $h>0$ gives the IDE
	$$
		u_{t+1}(x)=u_t(x)+h\int_\Omega f(ht,x,y,u_t(y))\d y\fall(t,x)\in\I'\tm\Omega
	$$
	matching \eqref{ide} with a compact $\Omega\subset\R^\kappa$ and the Lebesgue measure $\mu$. 

	\item Population genetics or ecological models of the form
	$$
		u_{t+1}(x)=(1-\vartheta)g(x,u_t(x))+\vartheta\int_\Omega f(x,y,u_t(y))\d y
		\fall(t,x)\in\I'\tm\Omega
	$$
	are investigated in \cite{volkov:lui:07}, where $\vartheta\in[0,1]$ is a parameter and e.g.\ continuous functions $g:\Omega\tm\R^d\to\R^d$, $f:\Omega^2\tm\R^d\to\R^d$. These problems are of the from \eqref{ide} with a compact $\Omega\subset\R^\kappa$ and the Lebesgue measure $\mu$. 

	\item Let the compact set $\Omega\subset\R^\kappa$ be countable, $\eta\in\Omega$ and $w_\eta\geq 0$ be reals. Then $\mu(\Omega'):=\sum_{\eta\in\Omega'}w_\eta$ defines a measure on the family of all countable subsets $\Omega'\subset\R^\kappa$. The assumption $\sum_{\eta\in\Omega}w_\eta<\infty$ ensures that $\mu(\Omega)<\infty$. W.r.t.\ the resulting $\mu$-integral $\int_\Omega u\d\mu=\sum_{\eta\in\Omega}w_\eta u(\eta)$ the IDE \eqref{ide} becomes
	$$
		u_{t+1}(x)
		=
		g_t(x,u_t(x))
		+
		\sum_{\eta\in\Omega}w_\eta k_t(x,\eta,u_t(\eta))
		\fall(t,x)\in\I'\tm\Omega.
	$$
	Such difference equations occur as \emph{Nystr\"om methods} with \emph{nodes} $\eta$ and \emph{weights} $w_\eta$ as used in numerical discretizations and simulations \cite{atkinson:92} of IDEs \eqref{ide}. 
\end{itemize} 

\noindent
\textbf{Hypothesis}: 
For every $t\in\I'$ we suppose: 
\begin{itemize}
	\item[$(H_1)$] The function $g_t:\Omega\tm Z\to\R^d$ is such that $g_t(\cdot,z):\Omega\to\R^d$ is continuous and there exist reals $\gamma_t,\ell_t\geq 0$ with
	\begin{align*}
		\abs{g_t(x,z)}&\leq\gamma_t,&
		\abs{g_t(x,z)-g_t(x,\bar z)}&\leq\ell_t\abs{z-\bar z}\fall x\in\Omega,\,z,\bar z\in Z.
	\end{align*}

	\item[$(H_2)$] The kernel function $k_t:\Omega\tm\Omega\tm Z\to\R^d$ is such that $k_t(x,\cdot,z):\Omega\to\R^d$ is measurable for all $x\in\Omega$, $z\in Z$, and the following holds for almost all $y\in\Omega$: $k_t(x,y,\cdot):Z\to\R^d$ is continuous for all $x\in\Omega$ and the limit
	$$
		\lim_{x\to x_0}\int_\Omega\sup_{z\in Z\cap\bar B_r(0)}\abs{k_t(x,y,z)-k_t(x_0,y,z)}\d\mu(y)=0
		\fall r>0
	$$
	holds uniformly in $x_0\in\Omega$. 
	
	\item[$(H_3)$] There exists a function $\kappa_t:\Omega^2\to\R_+$, measurable in the second argument with $\sup_{x\in\Omega}\int_\Omega\kappa_t(x,y)\d\mu(y)<\infty$ and for almost all $y\in\Omega$ one has
	$$
		\abs{k_t(x,y,z)}\leq\kappa_t(x,y)\fall x\in\Omega,\,z\in Z. 
	$$
\end{itemize}
Then the Nemytskii operator $\eG_t:U\to C(\Omega,\R^d)$ satisfies
\begin{align}
	\norm{\eG_t(u)}_0&\leq\gamma_t,&
	\norm{\eG_t(u)-\eG_t(\bar u)}_0&\leq\ell_t\norm{u-\bar u}_0
	\fall u,\bar u\in U, 
	\label{noA}
\end{align}
while the Urysohn operators $\eK_t:U\to C(\Omega,\R^d)$ are globally bounded by 
\begin{align}
	\norm{\eK_t(u)}_0\leq\rho_t:=\sup_{x\in\Omega}\int_\Omega\kappa_t(x,y)\d\mu(y)\fall u\in U
	\label{noB}
\end{align}
and completely continuous due to \cite[p.~166, Prop.~3.2]{martin:76}. 

In the following we tacitly suppose $\eF_t(U)\subseteq U$ for all $t\in\I'$. 
\begin{proposition}[dissipativity for \eqref{ide}]\label{propdiss}
	If $(H_1$--$H_3)$ with
	\begin{align}
		\sup_{t\in\I'}\gamma_t&<\infty,&
		\sup_{t\in\I'}\rho_t&<\infty
		\label{nodiss1}
	\end{align}
	hold, then the bounded and closed set
	\begin{align*}
		\cA&:=\set{(t,u)\in\cU:\,\norm{u}_0\leq\gamma_{t^\ast}+\rho_{t^\ast}},&
		t^\ast&:=
		\begin{cases}
			t-1,&t>\min\I,\\
			t,&t=\min\I
		\end{cases}
	\end{align*}
	is positively invariant, forward absorbing (if $\I$ is unbounded above), pullback absorbing (if $\I$ is unbounded below) w.r.t.~\eqref{ide} with absorption time $1$. 
\end{proposition}
\begin{proof}
	Clearly the set $\cA$ is closed and due to \eqref{nodiss1} also bounded. Let $t,\tau\in\I$ with $\tau<t$. Thus, $t^\ast=t-1$ and our assumptions readily imply that
	\begin{eqnarray*}
		\norm{\vphi(t;\tau,u)}
		& \stackrel{\eqref{process}}{=} &
		\norm{\eF_{t-1}(\vphi(t-1;\tau,u))}\\
		&\stackrel{\eqref{rhsadd}}{\leq} &
		\norm{\eG_{t-1}(\vphi(t-1;\tau,u))}+\norm{\eK_{t-1}(\vphi(t-1;\tau,u))}\\
		& \stackrel{\eqref{noA}}{\leq} &
		\gamma_{t-1}+\norm{\eK_{t-1}(\vphi(t-1;\tau,u))}
		\stackrel{\eqref{noB}}{\leq}
		\gamma_{t^\ast}+\rho_{t^\ast}\fall(\tau,u)\in\cU
	\end{eqnarray*}
	and consequently $\vphi(t,\tau,\cU(\tau))\subseteq\cA(t)$ holds for all $\tau<t$. Thanks to $\cA,\cB\subseteq\cU$ for any bounded $\cB$ this inclusion guarantees that $\cA$ is positively invariant, but also forward and uniformly absorbing with absorption time $S=1$. 
	\qed
\end{proof}

\begin{theorem}[pullback attractor for \eqref{ide}]\label{thm52}
	Let $\I$ be unbounded below. If $(H_1$--$H_3)$ are satisfied with \eqref{nodiss1} and there exists a $T\in\I$ such that $\prod_{s=-\infty}^{T-1}\ell_s=0$ hold, then the IDE \eqref{ide} has a unique and bounded pullback attractor $\cA^\ast\subseteq\cA$. 
\end{theorem}
\begin{proof} 
	We aim to apply \tref{thmpullback} to \eqref{ide}. Thereto, \pref{propdiss} guarantees that \eqref{ide} is uniformly pullback absorbing. Moreover, since the Lipschitz constant $\ell_t$ of the Nemytskii operator $\eG_t$ is an upper bound for its Darbo constant and because $\eK_t$ is completely continuous, the assertion follows. 
	\qed
\end{proof}
Without further assumptions not much can be said about the detailed structure of the pullback attractor $\cA^\ast$. Nevertheless, in case the functions $g_t,k_t$ satisfy monotonicity assumptions in the second resp.\ third argument, it is possible to construct ``extremal'' solutions in the attractor \cite{poetzsche:15}. We illustrate this in
\begin{figure}
	\includegraphics[width=60mm]{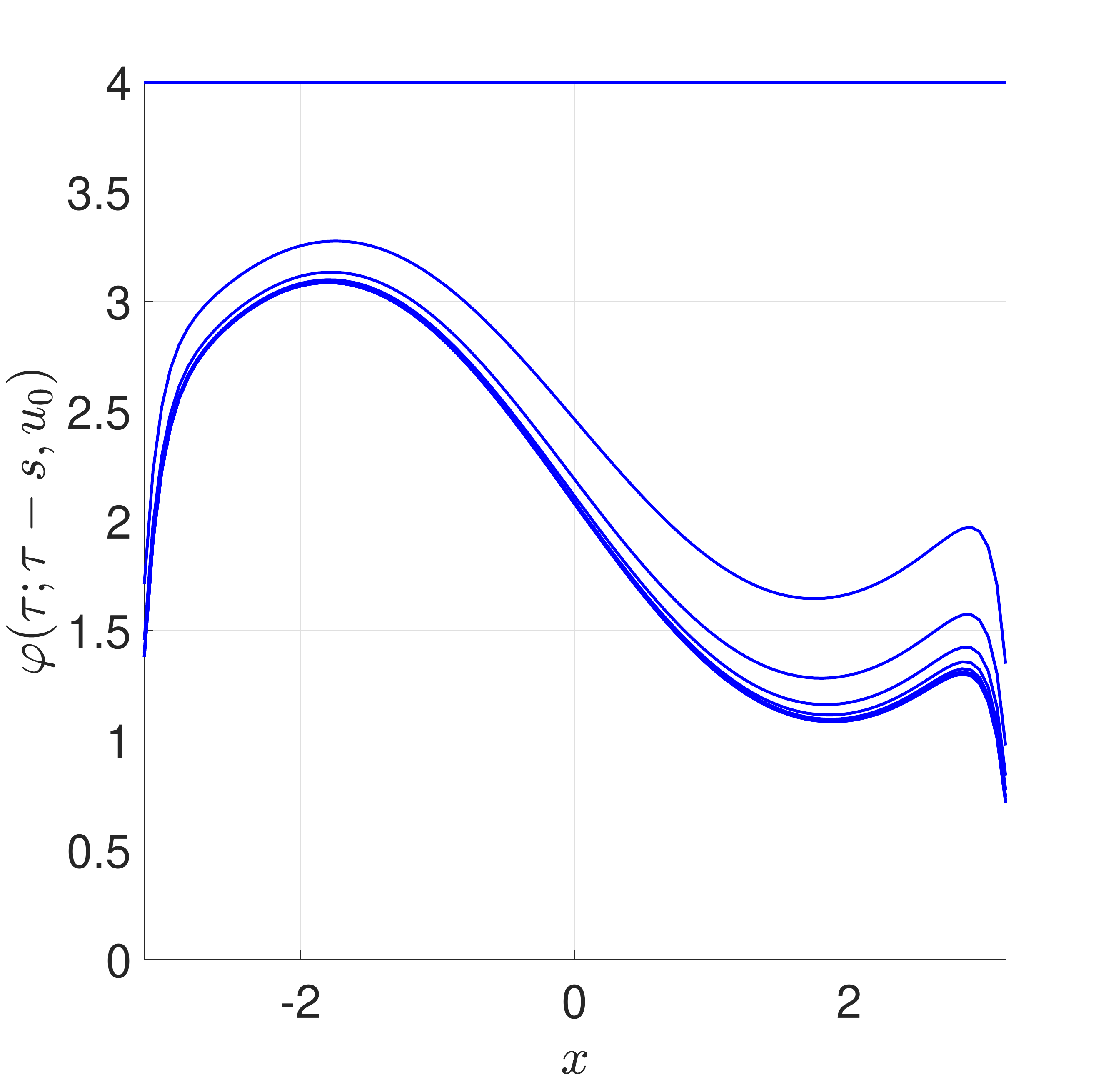}
	\includegraphics[width=60mm]{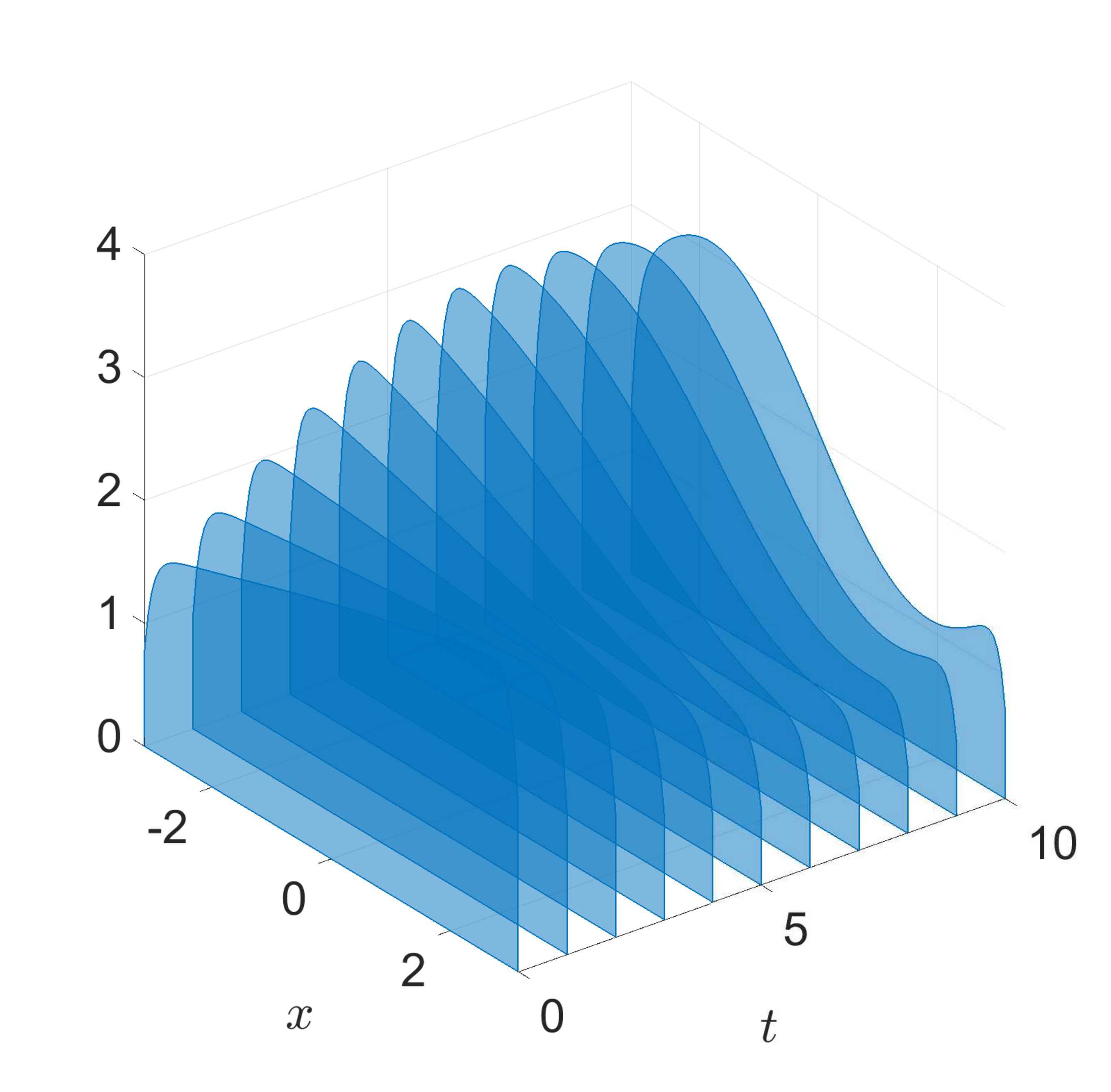}
	\caption{Pullback convergence to the fibre $\phi_\tau^+:\Omega\to\R_+$ ($\tau=10$, initial function $u_0(x)\equiv 4$, left) and sequence of sets containing the pullback attractor $\cA^\ast$ ($0\leq t\leq 10$, right) for $\vartheta=\tfrac{1}{4}$.}
	\label{figpullback1}
\end{figure}
\begin{example}[spatial Beverton-Holt equation]
	Let $\vartheta\in[0,1]$ and $a_t:\Omega\to(0,\infty)$, $t\in\I'$, be continuous functions describing the space- and time-dependent growth rates and a compact habitat $\Omega$. The spatial Beverton-Holt equation
	\begin{equation}
		u_{t+1}(x)
		=
		(1-\vartheta)
		\frac{a_t(x)u_t(x)}{1+u_t(x)}
		+
		\vartheta
		\int_\Omega k(x,y)\frac{a_t(y)u_t(y)}{1+u_t(y)}\d\mu(y)
		\label{exbh01}
	\end{equation}
	for all $(t,x)\in\I'\tm\Omega$ fits into the framework of \eqref{ide} with $Z=\R_+$, $U=C(\Omega,\R_+)$, 
	\begin{align*}
		g_t:\Omega\tm\R_+&\to\R_+,&g_t(x,z)&:=\frac{a_t(x)z}{1+z},\\
		k_t:\Omega\tm\Omega\tm\R_+&\to\R_+,&k_t(x,y,z)&:=k(x,y)\frac{a_t(y)z}{1+z}
	\end{align*}
	and a continuous kernel function $k:\Omega\tm\Omega\to(0,\infty)$. Then $(H_1$--$H_3)$ hold with 
	\begin{align*}
		\gamma_t&=(1-\vartheta)\alpha_t,&
		\ell_t&=(1-\vartheta)\alpha_t,&
		\kappa_t(x,y)&=\vartheta\alpha_tk(x,y)\fall x,y\in\Omega
	\end{align*}
	and $\alpha_t:=\max_{x\in\Omega}a_t(x)$. If $\lim_{t\to-\infty}(1-\vartheta)^{T-t}\prod_{s=t}^{T-1}\alpha_s=0$ holds for some $T\in\I$, then \tref{thm52} yields the existence of a pullback attractor $\cA^\ast\subseteq\cU$ for \eqref{exbh01}. Since the functions $g_t(x,\cdot),k_t(x,y,\cdot):\R_+\to\R_+$ are strictly increasing more can be said on the structure of $\cA^\ast$. As in \cite[Prop.~8]{poetzsche:15}, there exists an ``extremal'' entire solution $\phi^+$ being pullback attracting from above such that
	$$
		\cA^\ast\subseteq\set{(t,u)\in\cU:\,0\leq u(x)\leq\phi_t^+(x)\text{ for all }x\in\Omega}. 
	$$
	We illustrate both the pullback convergence to the solution $\phi^+$, as well as the sets containing solutions in the pullback attractor $\cA^\ast$ in \fref{figpullback1}, where $\Omega=[-\pi,\pi]$ is equipped with the $1$-dimensional Lebesgue measure, $a_t(x):=3-\sin\tfrac{tx}{10}$ (artificial) and the Laplace kernel $k(x,y):=\tfrac{a}{2}e^{-a\abs{x-y}}$ for the dispersal rate $a=10$. 
\end{example}

The remaining section addresses forward attraction. For simplicity we restrict to the class of Urysohn IDEs
\begin{equation}
	\tag{$I_0$}
	\boxed{u_{t+1}(x)
	=
	\int_\Omega k_t(x,y,u_t(y))\d\mu(y)}
	\fall(t,x)\in\I'\tm\Omega.
	\label{ide0}
\end{equation}

\noindent
\textbf{Hypothesis}: 
For every $t\in\I'$ we suppose: 
\begin{itemize}
	\item[$(H_4)$] For all $r>0$ there exists a function $\lambda_t:\Omega^2\to\R_+$, measurable in the second argument with $\sup_{x\in\Omega}\int_\Omega\lambda_t(x,y)\d\mu(y)<\infty$ and for almost all $y\in\Omega$ one has
	$$
		\abs{k_t(x,y,z)-k_t(x,y,z)}\leq\lambda_t(x,y)\abs{z-\bar z}
		\fall x\in\Omega,\,z,\bar z\in Z\cap\bar B_r(0). 
	$$
\end{itemize}

\begin{proposition}[dissipativity for \eqref{ide0}]\label{prop54}
	If $(H_2$--$H_3)$ with $R:=\sup_{t\in\I}\rho_t<\infty$ hold, then the bounded and compact nonautonomous set
	\begin{align*}
		\cA&:=\bigl\{(t,u)\in\cU:\,u\in\overline{\eK_{t^\ast}(U\cap B_R(0))}\bigr\},&
		t^\ast&:=
		\begin{cases}
			t-1,&t>\min\I,\\
			t,&t=\min\I
		\end{cases}
	\end{align*}
	is positively invariant, forward absorbing w.r.t.\ \eqref{ide0} with absorption time $2$. 
\end{proposition}
\begin{proof}
	Let $t\in\I'$ with $t-1\in\I'$ and thus $t^\ast=t-1$. Since the Urysohn operators $\eK_{t-1}$ are completely continuous, the fibres $\cA(t)=\overline{\eK_{t-1}(U\cap B_R(0))}$ are compact. Thanks to 
	$$
		\norm{\eK_t(u)}
		\stackrel{\eqref{noB}}{\leq}
		\rho_t\leq R\fall u\in U
	$$
	it follows that $\cA$ is bounded. Moreover, $\eK_t(\cA(t))\subseteq\eK_t(U\cap B_R(0))=\cA(t+1)$ holds for all $t\in\I'$ and $\cA$ is positively invariant. Furthermore, from the inclusion
	$$
		\vphi(t;\tau,u)
		\stackrel{\eqref{process}}{=}
		\eK_{t-1}(\underbrace{\vphi(t-1;\tau,u)}_{\in U\cap B_R(0)})
		\in
		\cA(t)\fall t-\tau\geq 2,\,u\in U
	$$
	we deduce that $\cA$ is absorbing. 
	\qed
\end{proof}

\begin{theorem}[forward limit set for \eqref{ide0}]
	Suppose that $(H_2$--$H_3)$ hold with additionally $R:=\sup_{t\in\I}\rho_t<\infty$. If $\bigcup_{t\in\I}\eK_t(U\cap B_R(0))$ is relatively compact and $\cA$ is the forward absorbing set from \pref{prop54}, then the following are true:
	\begin{enumerate}
		\item $\omega_{\cA}^+$ is asymptotically positively invariant, 

		\item $\omega_{\cA}^+$ is asymptotically negatively invariant, provided $(H_4)$ is satisfied with
		\begin{equation}
			\sup_{\tau\leq t<\tau+T}\prod_{s=\tau}^{t-1}
			\int_\Omega\lambda_s(x,y)\d\mu(y)<\infty
			\fall\tau\in\I,\,T\in\N.
			\label{thm55a}
		\end{equation}
	\end{enumerate}
\end{theorem}
The relative compactness of the union $\bigcup_{t\in\I}\eK_t(U\cap B_R(0))$ holds for instance, if the kernel functions $k_t$ stem from a finite set or the images $\eK_t(U\cap B_R(0))$ form a nonincreasing/nondecreasing sequence of sets. 
\begin{proof} %\cite[Thm.~2.2]{kloeden:16}
	By assumption the set $K:=\overline{\bigcup_{t\in\I}\eK_t(U\cap B_R(0))}$ is compact and this implies that \eqref{ide0} is strongly $\cA$-asymptotically compact. 
	
	(a) By construction of $K\subseteq U$ the assertion results from \tref{thm49}. 

	(b) Let $u,\bar u\in\cA(s)\cup K$ and choose $r>0$ so large that $\cA(s)\cup K\subseteq B_r(0)$. We conclude
	\begin{align*}
		\abs{\eK_s(u)(x)-\eK_s(\bar u)(x)}
		&\leq
		\int_\Omega\abs{k_s(x,y,u(y))-k_s(x,y,\bar u(y))\d\mu(y)}\\
		&\leq
		\int_\Omega\lambda_s(x,y)\d\mu(y)\norm{u-\bar u}_0\fall(s,x)\in\I\tm\Omega
	\end{align*}
	from assumption $(H_4)$. After passing to the least upper bound over $x\in\Omega$ it follows that $\sup_{x\in\Omega}\int_\Omega\lambda_s(x,y)\d\mu(y)$ is a Lipschitz constant for $\eK_s$ on $\cA(s)\cup K$. Hence, the assumption \eqref{thm55a} implies \eqref{thm47a} and therefore \tref{thm410} yields the claim. 
	\qed
\end{proof}

The above results do apply to the following
\begin{example}[spatial Ricker equation]
	Suppose that the compact $\Omega\subseteq\R^\kappa$ is equipped with the $\kappa$-dimensional Lebesgue measure $\mu$ and that $\mu(\Omega)>0$ holds. Let $(\alpha_t)_{t\in\I'}$ denote a bounded sequence of positive reals, $k:\Omega\tm\Omega\to\R_+$ be continuous and $(b_t)_{t\in\I'}$ be a bounded sequence in $C(\Omega,\R_+)$, $t\in\I'$. The spatial Ricker equation
	\begin{equation}
		u_{t+1}(x)
		=
		\alpha_t
		\int_\Omega k(x,y)u_t(y)e^{-u_t(y)}\d y+b_t(x)
		\fall(t,x)\in\I'\tm\Omega
		\label{exrick1}
	\end{equation}
	fits in the framework of \eqref{ide0} with $Z=\R_+$ and the kernel function
	$$
		k_t(x,y,z):=\alpha_tk(x,y)ze^{-z}+\tfrac{b_t(x)}{\mu(\Omega)}\fall x,y\in\Omega,\,z\in\R_+;
	$$
	hence, \eqref{exrick1} is defined on the cone $U:=C(\Omega,\R_+)$. If $\I$ is unbounded below, then \eqref{exrick1} possesses a pullback attractor $\cA^\ast\subseteq\cU$; see \fref{figpullback2} for an illustration. 
	
	For our subsequent analysis it is convenient to set $\gamma:=\sup_{x\in\Omega}\int_\Omega k(x,y)\d y$. We begin with some preparatory estimates. Above all, \eqref{exrick1} satisfies the assumption $(H_4)$ with $\lambda_t(x,y)=\alpha_tk(x,y)$, which guarantees the global Lipschitz condition
	\begin{equation}
		\norm{\eK_t(u)-\eK_t(\bar u)}
		\leq
		\alpha_t\gamma\norm{u-\bar u}\fall u,\bar u\in U.
		\label{exrick3}
	\end{equation}
	If we represent the right-hand side of \eqref{exrick1} in semilinear form \eqref{rhssemi} with
	\begin{align*}
		\eL_t u&:=\alpha_t\int_\Omega k(\cdot,y)u(y)\d y,&
		\eN_t(u)&:=\alpha_t\int_\Omega k(\cdot,y)(e^{-u(y)}-1)u(y)\d y+b_t, 
	\end{align*}
	then $\norm{\eL_t}=\alpha_t\gamma$ holds, as well as the global Lipschitz condition
	\begin{equation}
		\norm{\eN_t(u)-\eN_t(\bar u)}_0
		\leq
		\alpha_t(1+\tfrac{1}{e^2})\gamma\norm{u-\bar u}_0
		\fall u,\bar u\in U.
		\label{exrick5}
	\end{equation}
	
	In order to obtain information on the forward attractor, we suppose that $\I=\Z$ and that \eqref{exrick1} is asymptotically autonomous in forward time, i.e.,\ there exist $\alpha_+>0$ and $b\in C(\Omega,\R_+)$ such that
	\begin{align*}
		\lim_{t\to\infty}\alpha_t&=\alpha_+\in[0,1),&
		\lim_{t\to\infty}b_t&=b. 
	\end{align*}
	If $\alpha_+\gamma<1$, then it follows from the contraction mapping principle and \eqref{exrick3} that the autonomous limit equation
	\begin{equation}
		u_{t+1}(x)
		=
		\alpha_+
		\int_\Omega k(x,y)u_t(y)e^{-u_t(y)}\d y+b(x)
		\fall(t,x)\in\I'\tm\Omega
		\label{exrick7}
	\end{equation}
	has a unique, globally attractive fixed-point $u^\ast\in U$. 
	
	We choose $\I=\N_0$ and an absorbing set $A\subseteq U$ of the limit equation \eqref{exrick7} such that $\cA=\I\tm A$ is forward absorbing w.r.t.\ \eqref{exrick1}. If we assume $\sup_{s\leq t}\prod_{r=s}^{t-1}\frac{\alpha_r}{\alpha_+}\leq K$, then the growth estimate \eqref{thmaasl1} holds with $\alpha=\alpha_+\gamma$ due to
	$$
		\norm{\Phi(t,s)}
		\leq
		\prod_{r=s}^{t-1}\norm{\eL_r}
		=
		(\alpha_+\gamma)^{t-s}\prod_{r=s}^{t-1}\tfrac{\alpha_r}{\alpha_+}
		\leq
		K(\alpha_+\gamma)^{t-s}\fall s\leq t. 
	$$
	It follows from \eqref{exrick5} that every nonlinearity $\eN_t:U\to C(\Omega)$ has the Lipschitz constant $L:=(1+\tfrac{1}{e^2})\gamma\sup_{t\geq 0}\alpha_t$. Consequently, if furthermore $(\alpha_t)_{t\in\I}$ converges exponentially to $\alpha_+$ with rate $\alpha$, then \tref{thmaasl} applies under the assumption
	\begin{equation}
		(1+\tfrac{1}{e^2})\gamma\sup_{t\geq 0}\alpha_t<\tfrac{1-\alpha}{K}
		\label{exrick8}
	\end{equation}
	and thus \eqref{thmaa1} holds. Hence, we derive from \tref{thmaa} the relations
	$$
		\Omega_{\cA}(\tau)=\omega_\cA^-=\omega_\cA^+=\set{u^\ast}\fall\tau\in\N_0.
	$$
	If we concretely define 
	\begin{align*}
		\alpha_t&:=
		\begin{cases}
			\alpha_+(1+\alpha^t),&t\geq 0,\\
			\alpha_-,&t<0,
		\end{cases}&
		b_t&:=
		\begin{cases}
			b,&t\geq 0,\\
			0,&t<0,
		\end{cases}
	\end{align*}
	then this implies the elementary estimate
	\begin{align*}
		\prod_{r=s}^{t-1}\frac{\alpha_r}{\alpha_+}
		=
		\exp\intoo{\sum_{r=s}^{t-1}\ln(1+\alpha^r)}
		\leq
		\exp\intoo{\sum_{r=s}^{t-1}\alpha^r}
		\leq
		\exp\intoo{\tfrac{1}{1-\alpha}}\fall s\leq t.
	\end{align*}
	Hence, we set $K:=\exp\intoo{\tfrac{1}{1-\alpha}}$ and \eqref{exrick8} simplifies to $2(1+\tfrac{1}{e^2})\alpha\exp\intoo{\tfrac{1}{1-\alpha}}<1-\alpha$. This assumption can be fulfilled for $\alpha=\alpha_+\gamma$ sufficiently close to $0$, which requires the kernel data $\gamma$ or the asymptotic growth rate $\alpha_+$ to be small. 

	Even more concretely, on the habitat $\Omega=[-L,L]$ with some real $L>0$ we again consider the Laplace kernel $k(x,y):=\tfrac{a}{2}e^{-a\abs{x-y}}$ with dispersal rate $a>0$, which yields $\gamma=1-e^{-aL}$. In this framework, an illustration of the forward limit set $\omega_\cA^+=\set{u^\ast}$ and subfibres of the pullback attractor $\cA^\ast$ is given in \fref{figpullback2}. Here, both the pullback attractor $\cA^\ast$ and the forward limit set $\omega_\cA^+$ capture the long term behaviour of \eqref{exrick1}. 
	\begin{SCfigure}
		\includegraphics[width=60mm]{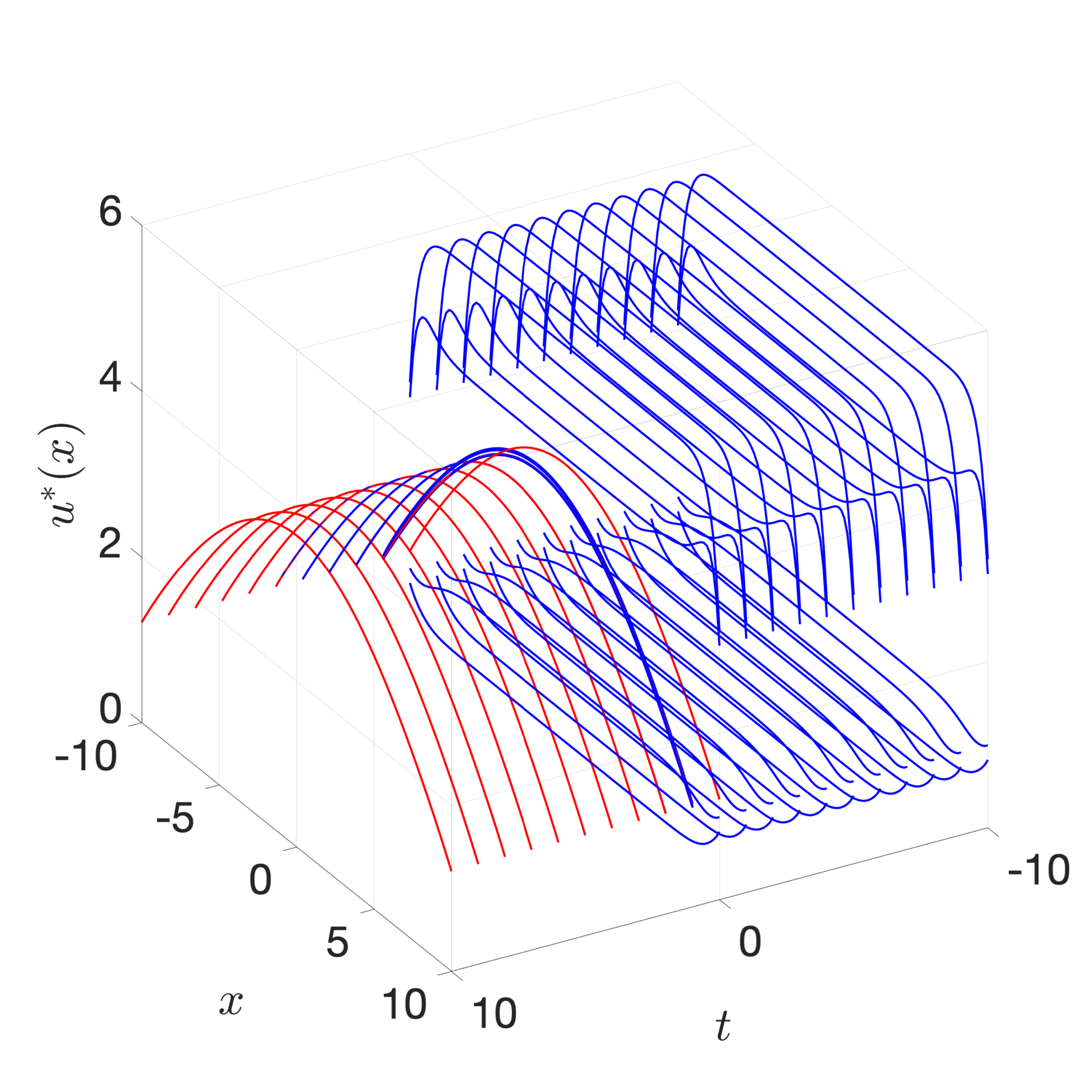}
		\caption{Functions contained in the fibres $\cA^\ast(t)$ of the pullback attractor over the times $-10\leq t\leq 10$ (blue) and the forward limit set $\N_0\tm\omega_{\cA}^+=\N_0\tm\set{u^\ast}$ (red) for the spatial Ricker equation \eqref{exrick1} with Laplace kernel ($a=2$, $L=10$), $\alpha_+=0.12$, $\alpha_-=14$ and the constant inhomogeneity $b(x):=5\cos\tfrac{x}{8}$
		\newline
		More detailed, depicted are the $4$-periodic orbits (blue) of the spatial Ricker equation, which is autonomous for $t<0$. In addition, the fibres $\cA^\ast(t)$ also contain $0$, a nontrivial fixed point and a $2$-periodic orbit. 
		}		
		\label{figpullback2}
	\end{SCfigure}
\end{example}
%
%
%
%\bigskip
%\begin{center}
%	\textsc{Acknowledgement}
%\end{center}
%\medskip
%The authors thank both referees for their constructive and helpful remarks improving the first version of this paper. 
%%
%%
%%
%\end{appendix}
%
%
%
%\bibliographystyle{amsplain}
%\bibliography{../Bib}
%
\providecommand{\bysame}{\leavevmode\hbox to3em{\hrulefill}\thinspace}
\providecommand{\MR}{\relax\ifhmode\unskip\space\fi MR }
\providecommand{\MRhref}[2]{\href{http://www.ams.org/mathscinet-getitem?mr=#1}{#2}}
\providecommand{\href}[2]{#2}

\end{document}